\documentclass[12pt]{amsart}
\usepackage{amsmath,amssymb,amsbsy,amsfonts,latexsym,amsopn,amstext,cite,amsxtra,euscript,amscd,bm,mathabx,mathrsfs, abraces}
\usepackage{url}
\usepackage[colorlinks,linkcolor=blue,anchorcolor=blue,citecolor=blue,backref=page]{hyperref}
\usepackage{color}
\usepackage{graphics,epsfig}
\usepackage{graphicx}
\usepackage{float} 
\usepackage[english]{babel}
\usepackage{mathtools}
\usepackage{todonotes}
\usepackage{url}
\usepackage{colonequals}
\usepackage{nicefrac}
\usepackage{enumitem}
\usepackage[colorlinks,linkcolor=blue,anchorcolor=blue,citecolor=blue,backref=page]{hyperref}

\usepackage[norefs,nocites]{refcheck}

\hypersetup{breaklinks=true}

\usepackage[norefs,nocites]{refcheck}
\usepackage[english]{babel}
\begin{document}

\newtheorem{theorem}{Theorem}
\newtheorem{thm}{Theorem}
\newtheorem{lem}[thm]{Lemma}
\newtheorem{claim}[thm]{Claim}
\newtheorem{cor}[thm]{Corollary}
\newtheorem{prop}[thm]{Proposition} 
\newtheorem{definition}[thm]{Definition}
\newtheorem{question}[thm]{Open Question}
\newtheorem{conj}[thm]{Conjecture}
\newtheorem{rem}[thm]{Remark}
\newtheorem{prob}{Problem}

\def\ccr#1{\textcolor{red}{#1}}
\def\ccm#1{\textcolor{magenta}{#1}}
\def\cco#1{\textcolor{orange}{#1}}

\def\vA{\mathbf A}
\def\vB{\mathbf B}

\newtheorem{ass}[thm]{Assumption}

\def \ss {s}

\newtheorem{lemma}[thm]{Lemma}

\newcommand{\sgn}{\operatorname{sgn}}
\newcommand{\rad}{\operatorname{rad}}
\newcommand{\GL}{\operatorname{GL}}
\newcommand{\SL}{\operatorname{SL}}
\newcommand{\lcm}{\operatorname{lcm}}
\newcommand{\ord}{\operatorname{ord}}
\newcommand{\Tr}{\operatorname{Tr}}
\newcommand{\Span}{\operatorname{Span}}

\numberwithin{equation}{section}
\numberwithin{theorem}{section}
\numberwithin{thm}{section}
\numberwithin{table}{section}

\def\vol {{\mathrm{vol\,}}}
\def\squareforqed{\hbox{\rlap{$\sqcap$}$\sqcup$}}
\def\qed{\ifmmode\squareforqed\else{\unskip\nobreak\hfil
\penalty50\hskip1em\null\nobreak\hfil\squareforqed
\parfillskip=0pt\finalhyphendemerits=0\endgraf}\fi}

\newcommand{\rank}{\operatorname{rank}}

\def \balpha{\bm{\alpha}}
\def \bbeta{\bm{\beta}}
\def \bgamma{\bm{\gamma}}
\def \blambda{\bm{\lambda}}
\def \bchi{\bm{\chi}}
\def \bphi{\bm{\varphi}}
\def \bpsi{\bm{\psi}}
\def \bomega{\bm{\omega}}
\def \btheta{\bm{\vartheta}}
\def \bmu{\bm{\mu}}
\def \bnu{\bm{\nu}}
\def \ba{\bm{a}}
\def \bb{\bm{b}}
\def \bs{\bm{s}}
\def \by{\bm{y}}
\def \bzero{\bm{0}}

\newcommand{\bfxi}{{\boldsymbol{\xi}}}
\newcommand{\bfrho}{{\boldsymbol{\rho}}}

\def\cA{{\mathcal A}}
\def\cB{{\mathcal B}}
\def\cC{{\mathcal C}}
\def\cD{{\mathcal D}}
\def\cE{{\mathcal E}}
\def\cF{{\mathcal F}}
\def\cG{{\mathcal G}}
\def\cH{{\mathcal H}}
\def\cI{{\mathcal I}}
\def\cJ{{\mathcal J}}
\def\cK{{\mathcal K}}
\def\cL{{\mathcal L}}
\def\cM{{\mathcal M}}
\def\cN{{\mathcal N}}
\def\cO{{\mathcal O}}
\def\cP{{\mathcal P}}
\def\cQ{{\mathcal Q}}
\def\cR{{\mathcal R}}
\def\cS{{\mathcal S}}
\def\cT{{\mathcal T}}
\def\cU{{\mathcal U}}
\def\cV{{\mathcal V}}
\def\cW{{\mathcal W}}
\def\cX{{\mathcal X}}
\def\cY{{\mathcal Y}}
\def\cZ{{\mathcal Z}}
\def\Ker{{\mathrm{Ker}}}

\def\NmQR{N(m;Q,R)}
\def\VmQR{\cV(m;Q,R)}

\def\Xm{\cX_m}

\def \A {{\mathbb A}}
\def \B {{\mathbb A}}
\def \C {{\mathbb C}}
\def \F {{\mathbb F}}
\def \G {{\mathbb G}}
\def \L {{\mathbb L}}
\def \K {{\mathbb K}}
\def \N {{\mathbb N}}
\def \P {{\mathbb P}}
\def \Q {{\mathbb Q}}
\def \R {{\mathbb R}}
\def \Z {{\mathbb Z}}

\def\fA{{\mathfrak A}}
 \def\fB{{\mathfrak B}}
\def\fC{{\mathfrak C}}
\def\fD{{\mathfrak D}}
\def\fG{{\mathfrak G}}
\def\fH{{\mathfrak H}}
\def\fI{{\mathfrak I}}
\def\fL{{\mathfrak L}}
\def\fM{{\mathfrak M}}
\def\fP{{\mathfrak P}}
\def\fR{{\mathfrak R}}
\def\fS{{\mathfrak  S}}
\def\fU{{\mathfrak U}}
\def\fY{{\mathfrak Y}}

\def \fUg{{\mathfrak U}_{\mathrm{good}}}
\def \fUm{{\mathfrak U}_{\mathrm{med}}}
\def \fV{{\mathfrak V}}
\def \fG{\mathfrak G}

\def\e{{\mathbf{\,e}}}
\def\ep{{\mathbf{\,e}}_p}
\def\eq{{\mathbf{\,e}}_q}

 \def\\{\cr}
\def\({\left(}
\def\){\right)}
\def\fl#1{\left\lfloor#1\right\rfloor}
\def\rf#1{\left\lceil#1\right\rceil}

\def\Im{{\mathrm{Im}}}

\def \oF {\overline \F}

\newcommand{\pfrac}[2]{{\left(\frac{#1}{#2}\right)}}

\def \Prob{{\mathrm {}}}
\def\e{\mathbf{e}}
\def\ep{{\mathbf{\,e}}_p}
\def\epp{{\mathbf{\,e}}_{p^2}}
\def\em{{\mathbf{\,e}}_m}

\def\Res{\mathrm{Res}}
\def\Orb{\mathrm{Orb}}

\def\vec#1{\mathbf{#1}}
\def \va{\vec{a}}
\def \vb{\vec{b}}
\def \vc{\vec{c}}
\def \vs{\vec{s}}
\def \vu{\vec{u}}
\def \vv{\vec{v}}
\def \vw{\vec{w}}
\def\vlam{\vec{\lambda}}
\def\flp#1{{\left\langle#1\right\rangle}_p}

\def\mand{\qquad\mbox{and}\qquad}

\title[Integer matrices with given determinant and height]
{A uniform formula on the number of integer matrices with given determinant and height}
\author{Muhammad Afifurrahman}\address{School of Mathematics and Statistics, University of New South Wales, Sydney NSW 2052, Australia}
\email{m.afifurrahman@unsw.edu.au}

\address[]{}
\email{}
\address[]{}
\email{}

\thanks{}
\begin{abstract} We obtain an asymptotic formula for the number of integer $2\times 2$ matrices that have determinant $\Delta$ and whose absolute values of the entries are at most $H$. The result holds uniformly for a large range of $\Delta$ with respect to $H$.
\end{abstract}

\maketitle

\tableofcontents
\section{Introduction} \subsection{Counting matrices} Recently there have been increasing interests in the arithmetic statistics of integer and rational matrices with entries of bounded height. For example, \cite{MeSh, MOS, HOS, OS2} consider the problem of counting integer matrices with bounded entries and fixed characteristic polynomial, \cite{ALPS} considers the eigenvalues of integer matrices, \cite{Afif} considers the product sets of  integer matrices, \cite{E-BLS} considers the rank of integer matrices modulo $p$, and \cite{BOS,BS} consider the problem of arithmetic statistics of $2\times 2$ integer matrices. Similar results have also been obtained in~\cite{AKOS} over rational matrices of bounded height.
Continuing this trend, we now consider a problem of counting $2\times 2$ integer matrices of bounded height with a fixed determinant.

First,  for a positive integer $H$ and an  integer $\Delta$, define \begin{align*}
    \cM_n(\Z;H)=\{A=(a_{i,j})_{i,j=1}^n \in \cM_{n}(\Z) \mid |a_{i,j}|\le H\text{ for }i,j=1,\dotsc,n\}
\end{align*}
and  \begin{align*}
\cD_n(H,\Delta)=\{A\in \cM_n(\Z;H) \mid \det A=\Delta \}.
\end{align*} 

We first consider general results for problems of bounding $\#\cD_n(H,\Delta)$. Duke, Rudnick and Sarnak~\cite[Example~1.6]{DRS} (for $\Delta\ne 0$) and Katznelson~\cite[Theorem~1]{K} (for $\Delta= 0$) give asymptotic formula for a matrix set that is similar to $\cD_n(H,\Delta)$, where the matrices are ordered using the $\ell^2$ matrix norm. Using their results, we have \begin{align}\label{eqn:DRSK}
    \#\cD_n(H,\Delta) \asymp \begin{cases}
        H^{n^2-n},&\qquad\text{ if }\Delta \ne 0,\\
        H^{n^2-n}\log H,&\qquad\text{ if }\Delta = 0.\\
    \end{cases}
\end{align}
 The related notations are defined in Section~\ref{sec:not}.

However, the result in~\eqref{eqn:DRSK} is not uniform in our setting, with respect to $H$. The best uniform upper bound on $\#\cD_n(H,\Delta)$ for general $n$ is due to   Shparlinski~\cite{S}, who proved \begin{align}\label{eqn:detS}
    \#\cD_n(H,\Delta) =O(H^{n^2-n}\log H).
\end{align} 
Improving~\eqref{eqn:detS} to match~\eqref{eqn:DRSK} would give many improvements to problems in arithmetics statistics of integer matrices, for example~\cite{Afif,HOS}.

We now consider the case $n=2$, to which there are a few other results for some values of $\Delta$.  First, for $\Delta=0$, the problem is equvalent to counting integer solutions to \begin{align*}
    ad=bc,
\end{align*} with $|a|,\:|b|,\:|c|,\:|d|\le H$. Ayyad, Cochrane and Zheng~\cite{ACZ} obtained results when all variables are required to be positive. As a corollary from their result, we obtain, in our notation,  \begin{align}\label{eqn:ACZ}\begin{split}
    \#\cD_2(H,0) &= \dfrac{96}{\pi^2} H^2\log H + (8C+16) H^2 \\ &\quad + O(H^{19/13}\log^{7/13}H),
\end{split}
\end{align} with $C$ being an explicit constant. Using methods in this paper, we partially recover this result (up to the main term) in Proposition~\ref{prop:0}. We note that another work of the author and Corrigan~\cite{AC} also recovers similar result, albeit without the explicit value of $C$ and different error term. Also, this problem is also related to many other matrices-related problems; we refer the reader to~\cite[Section~2.3]{AC} for further discussion.

For $\Delta=1$,  the set $\cD_2(H,1)$ is the set of elements in the $\SL_2(\Z)$ group whose absolute value of the entries are bounded by $H$. Related to this, the asymptotics of $\#\cD_2(H,1)$ is already known to Selberg, see~\cite[Corollary~15.12]{IK}, who used automorphic forms in the proof. Similar result can also be seen in~\cite{New}. 
However, the implied constants in these results depend on $H$; hence, these are not uniform results. 

In a related direction, Bulinski and Shparlinski~\cite{BS} considered the problem of finding an asymptotic formula for the cardinalities of the intersection of $\cM_2(\Z,H)$ and the congruence subgroups of  $\SL_2(\Z)$. 
Using their result, we have \begin{align}\label{eqn:BS}
\#\cD_2(H,1) =\dfrac{96}{\pi^2}H^2+O(H^{5/3+o(1)}).
\end{align}

For the case $\Delta=p$ where $p$ is a prime, Martin, White and Yip~\cite{MWY} counted the number of integer solutions to \begin{align*}
    ad+bc=p,
\end{align*} with $1\le a,b,c,d\le H$. This problem is closely related to our problem, after some sign considerations; see Sections~\ref{sec:sign} and~\ref{sec:11-1}. They used this result to provide an asymptotics on the number of direction of the set defined by $\{1,\dots,n\}^2$ over the set $\F_p^2$, with $\F_p$ being the finite field of $p$ elements.

In this work, we consider the problem of finding an asymptotic formula of $\#\cD_2(H,\Delta)$, as $\Delta,\: H\to \infty$, with $\Delta\ne 0$. The  formula should be uniform with respect to $\Delta$ or $H$ and hold for large intervals of $|\Delta|$ with respect to $H$.

\subsection{The main result}

We now state the main result of this work.
\begin{thm}\label{thm:main}
For $\Delta \ne 0$, we have \begin{align*}
	\#\cD_2(H,\Delta) = \dfrac{96}{\pi^2} \dfrac{\sigma(|\Delta|)}{|\Delta|} H^2 + O(H^{o(1)}\max(H^{5/3},|\Delta|))
\end{align*}  as $H,|\Delta|\to \infty$, where $\sigma(n)$ denotes the sum of all positive divisors of $n$. 
\end{thm}
This theorem gives an asymptotic uniform formula for $\#\cD_2(H,\Delta)$ when $0<|\Delta|\le H^{2-\varepsilon}$ for all $\varepsilon>0$, which improves~\eqref{eqn:detS} for $n=2$.

Due to the limitations of our method, we do not have an asymptotic formula for $\#\cD_2(H,\Delta)$ when $|\Delta|\ge H^{2-\varepsilon}$. 
In this case, Theorem~\ref{thm:main} only gives \begin{align*}
    \#\cD_2(H,\Delta)=O(H^{2+o(1)}),
\end{align*} which is already known from~\eqref{eqn:detS}.

As a corollary of Theorem~\ref{thm:main}, we have that, for a fixed $\Delta\neq 0$,  \begin{align}\label{eqn:5/3}
    \#\cD_2(H,\Delta)=\dfrac{96}{\pi^2} \dfrac{\sigma(|\Delta|)}{|\Delta|} H^2+O(H^{5/3+o(1)})
\end{align}  as $H\to \infty$.
One may check that this result matches \eqref{eqn:BS} when $\Delta=1$.

Other than counting the number of integer matrices, our result can also be used to compute the second moment of a restricted divisor function $\tau_N$, defined and explained in Section~\ref{sec:tau} and Appendix~\ref{apx}.

Our strategies in this paper are based on counting the number of solutions of the related determinant equation \begin{align*}
    ad=\Delta+bc
\end{align*} for $a$ and $c$ over some intervals, specified in Sections~\ref{sec:111} and~\ref{sec:11-1}. The main tools used in the proofs come from Ustinov's~\cite{Ust, UstAr} works on points on the number of points $(u,v)$ on some modular hyperbola with intervals specified by a function. Other important tools are some summation identities related to the GCD functions that we derive in Section~\ref{sec:sum}. Note that these tools are not typically employed for these families of problems.

In addition, we also obtain the following proposition for the case $\Delta=0$. Note that this statement is weaker than~\eqref{eqn:ACZ}, however the methods used are different.
\begin{prop}\label{prop:0}
As $H\to \infty$, we have \begin{align*}
	\#\cD_2(H,0) = \dfrac{96}{\pi^2}  H^2 \log H +O(H^2).
\end{align*}  
\end{prop}

After the first version of this work appears in July 2024, the author was informed about the works~\cite{GG, Guria} that proved when $\Delta\ne 0$ and $0<|\Delta|\ll H^{1/3}$, \begin{align}\label{eqn:3/2}
      \#\cD_2(H,\Delta)=\dfrac{96}{\pi^2} \dfrac{\sigma(|\Delta|)}{|\Delta|} H^2+
          O(|\Delta|^\theta H^{3/2+o(1)}),
\end{align}
where $\theta$ is an admissible exponent from generalized Ramanujan Conjecture. With respect of this conjecture, $\theta=0$ is conjectured to be an admissible exponent. The best proven exponent comes from the work of Kim and Sarnak~\cite[Appendix~2]{KS}, where one may take $\theta=7/64$.

The result in~\eqref{eqn:3/2} improves Theorem~\ref{thm:main} for $0<|\Delta|\ll H^{1/3}$.  Also, this result improves the error bound of~\eqref{eqn:5/3} to $3/2+7/192+o(1)$. In the same works, they also proved   Theorem~\ref{thm:main} but only for the range $0<|\Delta|\ll H^{5/3}$ .

The methods used in these works, which heavily uses modular forms and character sums, are fully different from ours. In addition, their methods can also be used for counting the number of $2\times 2$ matrices in a box with a fixed characteristic polynomial~\cite{G1}, and the number of $2\times 2$ integer matrices in a box with some of the variables are primes~\cite{G2}.

\subsection{Notations}\label{sec:not}
For a finite set $\cS$, we use $\# \cS$ to denote its cardinality.  For positive integers $q$ and $K$, we denote $$\delta_q(K)=\begin{cases}
	1,&\qquad q\mid K, \\ 0,&\qquad  q\nmid K.
\end{cases}$$
For a positive integer $n$, we denote $\tau(n)$ and $\sigma(n)$ as the number of positive divisors and sum of positive divisors of $n$, respectively. 

As usual we define
\[
\sgn u = \begin{cases} 
-1, & \qquad \text{if } u < 0,\\
0, & \qquad \text{if } u =0,\\
1, & \qquad \text{if } u >0.
\end{cases}
\]

The  equivalent 
notations
\[
U = O(V) \quad \Longleftrightarrow \quad U \ll V \quad \Longleftrightarrow \quad V\gg U
\] 
all mean that $|U|\le c V$ for some positive constant $c$ that may depends on $n$. We also denote \begin{align*}
	U \asymp V \iff U =O(V) \text{ and } V=O(U).
\end{align*} We also denote $o(1)$ for an expression in $x$ that goes to 0 as $x\to \infty$.

\section{Preliminary results}
\subsection{Distribution of points on a modular hyperbola}In this section, we recall some results on the distribution of points on the modular hyperbola \begin{align}\label{eqn:uvK}
	uv \equiv K \pmod{q},
\end{align} where $q\ge 1$ is an arbitrary integer.

We start with recalling an asymptotical formula on $N(K,q;U,V,X,Y)$, the number of integer solutions $(u,v)$  of Equation~\eqref{eqn:uvK} in a rectangular domain $[U, U+X]\times [V,V+Y]$. 
For example, such a result has been recorded in~\cite[Lemma~5]{UstAr}.

\begin{lemma} 
\label{lem:ModHyp-Box}  For any $U,V\ge 1$, we have 
\begin{align*}
	N(K,q;U,V,X,Y)= \dfrac{Y}{q} \sum_{r|K}\sum_{\substack{ U< u \le  U+X\\\gcd(u,q)=r}} r + E
\end{align*} where
\begin{align*}
    |E| \le q^{o(1)}(q^{1/2}+XDq^{-1}+D),
\end{align*}with $D=\gcd(K,q)$.
\end{lemma}
We note that the case $\gcd(K,q)=1$ of Lemma~\ref{lem:ModHyp-Box} is given in~\cite[Theorem~13]{Shp}. 

Next, we recall a result of Ustinov~\cite{Ust} on the number of points $(u,v)$ on the modular hyperbola \eqref{eqn:uvK} with \begin{align*}
U<u \le  U+X \qquad\text{and}\qquad 0< v \le  f(u),
\end{align*} for some positive function $f$ with a continuous second derivative. Let \begin{align*}
\cT_f(K,q;U,X) = \{&(u,v)\in \Z^2 \colon U < u \le  U+X, 0<v \le  f(u),\\& \quad uv \equiv K\pmod{q}\},  
\end{align*} and denote $T_f(K,q;U,X)=\#\cT_f(K,q;U,X)$. We have the following estimation from~\cite[Theorem~1]{Ust}.

\begin{lemma}\label{lem:Ustinov}
Assume that the function $f\colon \R \to \R_{\ge 0}$ has a continuous second derivative on $[U,U+X]$, and for some $L>0$ we have \begin{align*}
	|f''(u)|  \asymp \dfrac{1}{L}, \qquad u\in[U,U+X].
\end{align*} 
Then, \begin{align*}
	T_f(K,q;U,X)&=\dfrac{1}{q} \sum_{r|K}\sum_{\substack{ U< u \le  U+X\\\gcd(u,q)=r}} rf(u)-\dfrac{X\delta_q(K)}{2}+E,\end{align*} 
where\begin{align*}
	|E|\le 
	q^{o(1)}(XL^{-1/3}+q^{-1}D^{1/2}L^{1/2}+q^{1/2}+D),
\end{align*}
with $D=\gcd(K,q)$.
\end{lemma}

\subsection{A restricted divisor function}\label{sec:tau} We often work with the following modification of the usual divisor function $\tau$. For positive integers $N$ and $n$, define
\begin{align}\label{eqn:taux}
    \tau_N(n)\colonequals \#\{(a,b)\in \Z^2\colon ab=n,\:1\le a,b\le N\}.
\end{align}
The function $\tau_N$ first appeared in Truelsen~\cite{Truelsen}, under a slightly different notation, who used this function in relation the problem of determining the
irrational numbers $\alpha$ for which the pair correlation for the fractional parts of $n^2\alpha$ is Poissonian.
The same function also appears in Mastrostefano~\cite{M}, who used it to classify maximal size product sets of random subsets of $\Z$. 

With respect to our problem, we note that  $\#\cD_2(N;\Delta)$ is related to the number of integer solutions of \begin{align*}
    ad-bc=\Delta,
\end{align*} with $1\le a,b,c,d\le N$. If we fix $bc=n$ and add the terms as $1\le n\le N^2$, we note that $ \#\cD_2(N;\Delta)$ is related to the expression \begin{align}\label{eqn:tau}
  \sum_{n \in \N} \tau_{N}(n)\tau_N(n+\Delta)= \sum_{n=1}^{N^2} \tau_{N}(n)\tau_N(n+\Delta).
\end{align}
The exact relation between these expressions is discussed in Section~\ref{sec:sign} and Equation~\eqref{eqn:1-1}. We will give an asymptotical result for~\eqref{eqn:tau} in the end of this section.

Now, we discuss some preliminary properties of $\tau_{N}$. By using this definition and double counting, we have that, for a fixed $N$, \begin{align}\label{eqn:sumtaux}
\sum_{n\in \N} \tau_N(n)  =\sum_{n=1}^{N^2}  \tau_N(n) = N^2.
\end{align}
Furthermore, we note that the definition in~\eqref{eqn:taux} is equivalent to \begin{align}\label{eqn:nN}
    \tau_N(n)=\#\{d \in \Z\colon d|n,\: n/N\le d\le N\}.
\end{align}
Furthermore, an unpublished note from Mastrostefano~\cite{M2} studies the positive moments of $\tau_N$. In particular, from~\cite[Lemma~2.1]{M2}, there exist two positive constants $C_1\le C_2$ such that \begin{align*}
    C_1 N^2\log N \le \sum_{n=1}^{N^2}  (\tau_N(n))^2 \le C_2 N^2\log N.
\end{align*} In addition, a remark after~\cite[Lemma~2.1]{M2} also asserts that \begin{align*}
    \sum_{n=1}^{N^2}  \tau_N(n)^2 = (D+o(1))N^2\log N
\end{align*} as $N\to \infty$, see also~\cite[Theorem~7]{HB}.

Both of these results are improved in the following proposition, which is proven in Section~\ref{sec:0}. This proposition implies Corollary~\ref{cor:0}, which in turns complete the proof of Proposition~\ref{prop:0}.
\begin{prop}\label{prop:tau} As $N\to \infty$, we have 
     \begin{align*}
    \sum_{n=1}^{N^2}  \tau_N(n)^2 = \dfrac{12}{\pi^2}N^2\log N +O(N^2).
\end{align*} 
\end{prop}
We note that this proposition is also superseded by Ayyad, Cochrane and Zheng's result in~\cite{ACZ}, however we provide this proposition for the sake of completeness. 

Next, by using a result of Section~\ref{sec:11-1}, we have the following corollary on the shifted sum of $\tau_N$, which mirrors~\eqref{eqn:additivedivisor}. \begin{cor}\label{cor:tau}For a fixed $\Delta>0$ and $N\to \infty$, we have 
     \begin{align*}
    \sum_{n=1}^{N^2}  \tau_N(n)\tau_N(n+\Delta) = \dfrac{12}{\pi^2}\dfrac{\sigma(\Delta)}{\Delta}N^2\log N +O(N^{o(1)}\max(N^{5/3},\Delta)).
\end{align*} 
    
\end{cor}

For a further discussion of the function $\tau_N$, we refer the readers to Appendix~\ref{apx}. In particular, the appendix also contains the connection between~\eqref{eqn:tau} and the problem of binary additive divisor over $\tau$.

\subsection{Some summation formulas}\label{sec:sum} Here we derive some summation formulas over positive integers that we use in the next sections. We do not attempt to optimise the error terms in the given results, which are of separate interests.

First,  we derive the following estimation on a summation formula related to GCD functions. 
 \begin{lemma}\label{lem:gcdsum} Let $A$ and $B$ be real numbers with $B\le 1$, and $K$ and $L$ be positive integers. Then, we have
\begin{align*}
	\sum_{c=1}^K c^A\gcd(c,L)^B \le K^{A+1+o(1)}L^{o(1)},
\end{align*}
where the sum is taken over all positive integer $1\le c\le K$.
\end{lemma}
\begin{proof}
For a fixed positive integer $c$, we  set $\gcd(c,L)=d$ and rewrite $c=de$ for some $e\le K/d$. We have 
\begin{align*}
	\sum_{c=1}^K c^A\gcd(c,L)^B &\le \sum_{d|L}\sum_{e\le K/d} (de)^A d^B\\
	&= \sum_{d|L}d^{A+B}\sum_{e\le K/d} e^A\\
	&\le \sum_{d|L}d^{A+B}(K/d)^{A+1}\\
	&\le K^{A+1+o(1)}\sum_{d|L}d^{B-1+o(1)}\\
	&\le K^{A+1+o(1)}L^{o(1)}\sum_{d|L}1\\
	&=  K^{A+1+o(1)}L^{o(1)},
\end{align*} where we have used $B\le 1$ to conclude  $d^{B-1+o(1)}\le L^{o(1)}$. 
\end{proof} Next, we recall some well-known summation formulas involving the Euler totient function $\varphi(n)$, see for example~\cite[Chapter~3, Exercise~5-6]{Ap}.
\begin{lemma}\label{lem:sumvarphi} 
For $X>0$, we have \begin{align*}
	&\sum_{n=1}^X \dfrac{\varphi(n)}{n} = \dfrac{6}{\pi^2} X+O(\log X),\\
	&\sum_{n=1}^X \dfrac{\varphi(n)}{n^2} = \dfrac{6}{\pi^2}\log X+O(1).
\end{align*}
\end{lemma}
We also need to count numbers of positive integers not bigger than $X$ that are relatively prime to other integer $Y$. We provide the result with proof for the sake of completeness.\begin{lemma}\label{lem:sum2}
    For $X,\:Y>0$, we have \begin{align*}
        & \sum_{\substack{0<x\le X\\ \gcd(x,Y)=1}}1 = \dfrac{X}{Y}\varphi(Y)+O(\tau(Y)).
    \end{align*}
\end{lemma}
\begin{proof}
    Using the M\"obius function, we have \begin{align*}
        \sum_{\substack{0<x\le X\\ \gcd(x,Y)=1}}1 &= \sum_{0<x\le X} \sum_{d|\gcd(x,Y)}\mu(d) 
       \\&= \sum_{d|Y}\mu(d) \left\lfloor \dfrac{X}{d}\right\rfloor
       \\&= \dfrac{X}{Y} \sum_{d|Y}\mu(d) \dfrac{Y}{d} + O(\tau(Y))
       \\&=  \dfrac{X}{Y}\varphi(Y)+O(\tau(Y)).\quad \qedhere
    \end{align*}
\end{proof}
Based on Lemmas~\ref{lem:sumvarphi} and~\ref{lem:sum2}, we derive some  summation formulas over pairs of integers $x$, $y$ with $\gcd(x,y)=r$, for some fixed $r$. These formulas are of independent interests.
\begin{lemma}\label{lem:xy}
Let $X,Y\ge 0$ and $r$ be a positive integer. We have \begin{align*}
	&\sum_{\substack{0<x,y\le X\\\gcd(x,y)=r}} \dfrac{r}{xy} = \dfrac{6}{\pi^2} \dfrac{1}{r}\log^2 \dfrac{X}{r} + O\bigg(\dfrac{1}{r}\log\dfrac{X}{r} \bigg),\\
	&\sum_{\substack{0<y<x+Y\\0<x\le X\\\gcd(x,y)=r}} \dfrac{r}{x} =\dfrac{6}{\pi^2}\left(\dfrac{X}{r} +\dfrac{Y}{r}\log \dfrac{X}{r}\right) + O \bigg(\dfrac{Y}{r} \bigg)+O \bigg(\log^2 \dfrac{X}{r} \bigg),\\
 &\sum_{\substack{x+Y<y<X\\0<x\le X\\\gcd(x,y)=r}} \dfrac{r}{y} =\dfrac{6}{\pi^2}\left(\dfrac{X-Y}{r} +\dfrac{Y}{r}\log \dfrac{X}{Y}\right) + O \bigg(\dfrac{Y}{r} \bigg)+O \bigg(\log^2 \dfrac{X}{r} \bigg),\\
 & \sum_{\substack{0<y\le Y\\0<x\le X\\\gcd(x,y)=r}} \dfrac{r}{y}=\dfrac{6}{\pi^2}\dfrac{X}{r}\log\dfrac{Y}{r}+O\bigg(\dfrac{X}{r}\bigg).
\end{align*}\end{lemma}
\begin{proof}
	Denote $X/r=X'$.
For the first equation,	we first have 
	\begin{align*}
		\sum_{\substack{x,y\le X\\\gcd(x,y)=r}} \dfrac{r}{xy}
		&= \dfrac{1}{r}\sum_{\substack{x',y'\le X'\\\gcd(x',y')=1}} \dfrac{1}{x'y'}
		\\&= \dfrac{1}{r}\sum_{x',y'\le  X'} \sum_{d|\gcd(x',y')} \dfrac{\mu(d)}{x'y'}
		\\&= \dfrac{1}{r} \sum_{d\le X/r} \mu(d) \sum_{
			\substack{d|x',y'\\x',y'\le X'}} \dfrac{1}{x'y'}
		\\&= \dfrac{1}{r} \sum_{d\le X'} \bigg[ \dfrac{\mu(d)\log^2(X'/d)}{d^2}  + O \bigg(  \dfrac{\log(X'/d)}{d^2} \bigg) \bigg].
	\end{align*}
	
 For $D=\log X'$, we split the summation in the last expression to $d>D$ and $d\le D$.
 
 First, for $d>D$, we have \begin{align}\label{eqn:dD1}\begin{split}
 \dfrac{1}{r} \sum_{D<d\le X'} \bigg[ \dfrac{\mu(d)\log^2(X'/d)}{d^2}  + O \bigg(  \dfrac{\log(X'/d)}{d^2} \bigg) \bigg] 
 &\ll  \dfrac{1}{r} \sum_{D<d\le X'} \dfrac{\log^2 X'}{d^2} 
\\ &\ll \dfrac{1}{r} \cdot \dfrac{\log^2 X'}{D}
\\ &\ll\dfrac{\log X'}{r} .
\end{split}
 \end{align}
 
Next, for $d\le D$, we first note that \begin{align*}
	\log^2 (X'/d)= \log^2 X' + O(\log X' \log d)= \log^2 X' + O(\log X' \log \log X').
\end{align*} Therefore, we have 
\begin{align}\label{eqn:dD2}\begin{split}
	&\dfrac{1}{r} \sum_{d\le D} \bigg[ \dfrac{\mu(d)\log^2(X'/d)}{d^2}  + O \bigg(  \dfrac{\log(X'/d)}{d^2} \bigg) \bigg] 
	\\&\quad= \dfrac{1}{r} \sum_{d\le D} \bigg[ \dfrac{\mu(d)\log^2 X'}{d^2}  + O \bigg(  \dfrac{\log X' \log \log X'}{d^2} \bigg) \bigg] 
	\\&\quad=   \dfrac{\log^2 X'}{r}   \dfrac{6}{\pi^2} + O\bigg(\dfrac{1}{r}\dfrac{\log^2 X'}{D} \bigg) + O \bigg( \dfrac{1}{r} \dfrac{\log X' \log \log X'}{D} \bigg)
	\\&\quad= \dfrac{6}{\pi^2} \dfrac{\log^2 X'}{r} + O\bigg(\dfrac{\log X'}{r} \bigg).
	\end{split}
\end{align}
 
    Adding \eqref{eqn:dD1} and~\eqref{eqn:dD2} completes the proof for the first equation of Lemma~\ref{lem:xy}.

	For the second equation, with $Y'=Y/r$, we have 
\begin{align*}\begin{split}\sum_{\substack{0<y<x+Y\\0<x\le X\\\gcd(x,y)=r}} \dfrac{r}{x}&= \sum_{0<x'\le  X'} \sum_{\substack{0<y'\le x'+Y'\\\gcd(x',y')=1} } \dfrac{1}{x'}
	\\&=\sum_{\substack{0<x'\le  X'}} \left[\dfrac{x'+Y'}{x'}\dfrac{\varphi(x')}{x'} +O \left(\dfrac{\tau(x')}{x'}\right)\right]
		\\&= \dfrac{6}{\pi^2}X' +Y'\dfrac{6}{\pi^2}\log X' + O(Y') +O(\log^2 X').
		\end{split}
	\end{align*}
 Similarly, for the third equation,
\begin{align*}\begin{split}\sum_{\substack{x+Y<y\le X\\0<x\le X\\\gcd(x,y)=r}} \dfrac{r}{y}&= \sum_{\substack{0<x\le y-Y\\Y<y\le X\\\gcd(x,y)=r}} \dfrac{r}{y}
\\&=\sum_{Y'<y'\le  X'} \sum_{\substack{ 0<x'\le y-Y\\\gcd(x',y')=1} } \dfrac{1}{y'}
	\\&=\sum_{\substack{Y'<y'\le X'}} \left[\dfrac{y'-Y'}{y'}\dfrac{\varphi(y')}{y'} +O \left(\dfrac{\tau(y')}{y'}\right)\right]
		\\&= \dfrac{6}{\pi^2}(X'-Y') +Y'\dfrac{6}{\pi^2}\log \dfrac{X'}{Y'} + O(Y') +O(\log^2 X').
		\end{split}
	\end{align*}

 Finally, for the last equation, \begin{align*}
     \sum_{\substack{0<y\le Y\\0<x\le X\\\gcd(x,y)=r}} \dfrac{r}{y} &= \sum_{0<y'\le Y'}\sum_{\substack{0<x'\le X'\\\gcd(x',y')=1}}\dfrac{1}{y'}
     \\& =\sum_{0<y'\le Y'} \left[\dfrac{X'\varphi(y')}{y'^2}+O \bigg(\dfrac{\tau(y')}{y'}\bigg)\right]
     \\&= \dfrac{6}{\pi}X'\log Y'+ O(X').
 \end{align*}
\end{proof}

\section{Separating the main term}\label{sec:sign}
 We first observe that the bijection \begin{align*}
\begin{pmatrix}
	a&b\\c&d  
\end{pmatrix} \mapsto      \begin{pmatrix}
	c&d\\a&b
\end{pmatrix}
\end{align*} maps $\cD_2(H,\Delta)$ to $\cD_2(H,-\Delta)$. From now we may assume that $\Delta$ is nonnegative.

Next, we observe that for a fixed $\Delta\neq 0$, there are $H^{1+o(1)}$ matrices in 
$\cD_2(H,\Delta)$ which have at least a zero entry. Also, if $\Delta=0$, there are $O(H^2)$ matrices in $\cD_2(H,0)$ which have at least a zero entry. Hence, we can focus on the counting the matrices in $\cD_2(H,\Delta)$ with nonzero entries. 

We now consider the following eight sets for different choices of the signs of $a$, $c$, $d$: \begin{align*}
\cD_2^{\alpha,\gamma,\delta'}(H,\Delta)  = \bigg\{\begin{pmatrix}
	a&b\\c&d
\end{pmatrix}\in \cD_2(H,\Delta) \colon   \sgn a=\alpha, \sgn c = \gamma, \sgn d = \delta' \bigg\} \end{align*} with $\alpha$, $\gamma$, $\delta' \in \{-1,1\}$.

Next, we observe that the maps \begin{align*}
\begin{pmatrix}
	a&b\\c&d \end{pmatrix}  \mapsto \begin{pmatrix} -a &b\\c& -d 
\end{pmatrix} 
\end{align*} and \begin{align*}
\begin{pmatrix}
	a&b\\c&d 
\end{pmatrix} \mapsto \begin{pmatrix}a & -b\\-c&d      \end{pmatrix} 
\end{align*} preserve $\cD_2(H,\Delta)$. Therefore, we have \begin{align*}
\#\cD_2^{1,1,1}(H,\Delta) = \#\cD_2^{\alpha,\gamma,\alpha}(H,\Delta) 
\end{align*} and
\begin{align*}
\#\cD_2^{1,1,-1}(H,\Delta) = \#\cD_2^{\alpha,\gamma,-\alpha}(H,\Delta) 
\end{align*} for all pairs $\alpha$,$\gamma\in \{-1,1\}$.
This implies \begin{align}\label{eqn:sumdelta}\begin{split}
\#\cD_2(H,\Delta) &= 4[\#\cD_2^{1,1,1}(H,\Delta) +\#\cD_2^{1,1,-1}(H,\Delta) ]\\& \quad +\begin{cases}
    O(H^2),&\text{ if }\Delta=0,\\ O(H^{1+o(1)}),&\text{ if }\Delta \ne 0.
\end{cases}
\end{split}
\end{align}

As seen in~\eqref{eqn:sumdelta}, to prove Theorem~\ref{thm:main} and Proposition~\ref{prop:0}, it remains to calculate $\#\cD_2^{1,1,\pm 1}(H,\Delta) $. We first consider the case $\Delta=0$. In this case, the map \begin{align*}
\begin{pmatrix}
	a&b\\c&d 
\end{pmatrix} \mapsto \begin{pmatrix}a & b\\-c&-d      \end{pmatrix} 
\end{align*} is a bijection between $\cD_2^{1,1,1}(H,0)$ and $\cD_2^{1,1,-1}(H,0)$. Therefore, we may rewrite~\eqref{eqn:sumdelta} as \begin{align}\label{eqn:sum0}
    \#\cD_2(H,0)=8\#\cD_2^{1,1,1}(H,0)+O(H^2).
\end{align}
The following lemma provides the value of $\#\cD_2^{1,1,1}(H,0)$, which together with~\eqref{eqn:sum0} completes the proof of Proposition~\ref{prop:0}.
\begin{cor}\label{cor:0}
    As $H\to \infty$, we have \begin{align*}
    \#\cD_2^{1,1,1}(H,0)=\dfrac{12}{\pi^2} H^2\log H + O(H^2).
    \end{align*}
\end{cor}
This result is a corollary of Proposition~\ref{prop:tau}, to which we give the proof in Section~\ref{sec:0}.

Next, we consider the case $\Delta\ne 0$. In this case, we provide the value of $\#\cD_2^{1,1,\pm 1}(H,\Delta) $ in the following lemma.
\begin{lemma}\label{lem:sign}
For $\Delta>0$ we have \begin{align}
    \label{eqn:lem}
		\#\cD_2^{1,1,\pm 1}(H,\Delta) =&\dfrac{12}{\pi^2}H^2 \displaystyle \sum_{\substack{r|\Delta\\r\le H}} \dfrac{1}{r} +O(H^{o(1)}\max(H^{5/3},\Delta)).
\end{align}
\end{lemma}
The proof of this lemma is done in Sections~\ref{sec:111}~and~\ref{sec:11-1} by using Lemmas~\ref{lem:ModHyp-Box}~and~\ref{lem:Ustinov}, complemented with Lemmas~\ref{lem:gcdsum}~and~\ref{lem:xy}.

We then use this lemma to complete the proof of Theorem~\ref{thm:main}. After having~\eqref{eqn:lem}, we note that \begin{align}\label{eqn:full}
    \sum_{\substack{r|\Delta\\r\le H}} \dfrac{1}{r} =\sum_{\substack{r|\Delta}} \dfrac{1}{r} - \sum_{\substack{r|\Delta\\r>H}} \dfrac{1}{r} 
    = \dfrac{\sigma(\Delta)}{\Delta} + O(H^{-1}\Delta^{o(1)}).
\end{align}
We then substitute~\eqref{eqn:lem}~and~\eqref{eqn:full} to Equation~\eqref{eqn:sumdelta} to prove Theorem~\ref{thm:main} for $\Delta>0$. Finally, for negative $\Delta$, we have $$\#\cD_2(H,\Delta)=\#\cD_2(H,-\Delta),$$ which completes the proof of Theorem~\ref{thm:main}.

\section{Counting \texorpdfstring{$\#\cD_2^{1,1,1}(H,0)$}{D2111(H,0)}}\label{sec:0}
We recall from~\eqref{eqn:tau}, \begin{align*}
    \#\cD_2^{1,1,1}(H,0)=\sum_{n=1}^{H^2}(\tau_H(n))^2,
\end{align*}  where $\tau_H$ is defined in~\eqref{eqn:taux}.
  Next, from the property of $\tau_N$ in~\eqref{eqn:nN}, we have \begin{align}\label{eqn:tausquare}\begin{split}
    \sum_{n=1}^{H^2}(\tau_H(n))^2 &=\sum_{n=1}^{H^2} \left(\sum_{\substack{n/H \le d_1 \le H\\ d_1|n}}1 \right)^2 \\
    &= \sum_{n=1}^{H^2} \sum_{\substack{n/H \le d_1,d_2 \le H\\ d_1,d_2|n}}1 \\
    &=  \sum_{n=1}^{H^2} \left( \sum_{\substack{n/H \le d_1 \le H\\ d_1|n}}1 + 2\sum_{\substack{n/H \le d_1<d_2 \le H\\ d_1,d_2|n}}1 \right)\\
    &=  \sum_{n=1}^{H^2} \tau_H(n) + 2\sum_{n=1}^{H^2}\left( \sum_{\substack{n/H \le d_1<d_2 \le H\\ d_1,d_2|n}}1 \right)\\
    &=  H^2 + 2\sum_{n=1}^{H^2} \sum_{\substack{n/H \le d_1<d_2 \le H\\ d_1,d_2|n}}1,
\end{split}
\end{align} where the last equation is known from~\eqref{eqn:sumtaux}.

We now simplify the last line in~\eqref{eqn:tausquare}. We have \begin{align}\label{eqn:d1d2}\begin{split}    \sum_{n=1}^{H^2} \sum_{\substack{n/H \le d_1<d_2 \le H\\ d_1,d_2|n}}1 &= \sum_{0<d_1< d_2\le H} \sum_{\substack{n\le d_1H \\ \lcm(d_1,d_2)|n}} 1 \\
    &=\sum_{0<d_1< d_2\le H}  \dfrac{d_1H}{\lcm(d_1,d_2)} 
    \\&= H \sum_{0<d_1< d_2\le H}  \dfrac{\gcd(d_1,d_2)}{d_2} 
    \\&= H \sum_{r=1}^H \sum_{\substack{\gcd(d_1,d_2)=r\\0<d_1< d_2\le H}}  \dfrac{r}{d_2}. \end{split}\end{align}
    From Lemma~\ref{lem:xy}, we have
    \begin{align}\label{eqn:d1d2f}\begin{split}
      H \sum_{r=1}^H \sum_{\substack{\gcd(d_1,d_2)=r\\0<d_1< d_2\le H}}  \dfrac{r}{d_2}&= H \sum_{r=1}^H \left[ \dfrac{6}{\pi^2}\dfrac{H}{r} +O \left( \log^2\dfrac{H}{r}\right) \right]
    \\& = \dfrac{6}{\pi^2} H^2\log H + O(H^2).
\end{split}
\end{align} Combining~\eqref{eqn:tau},~\eqref{eqn:tausquare},~\eqref{eqn:d1d2}~and~\eqref{eqn:d1d2f} completes the proof of Proposition~\ref{prop:tau}, which implies Corollary~\ref{cor:0} and Proposition~\ref{prop:0}.

\section{Counting \texorpdfstring{$\#\cD_2^{1,1,1}(H,\Delta)$}{D2111(H,Delta)}}\label{sec:111}

\subsection{The case divisions}	\label{sec:casediv}
To prove Lemma~\ref{lem:sign}, we first count $\#\cD_2^{1,1,1}(H,\Delta)$, which is equivalent to counting the number of integer solutions to the equation \begin{align}\label{eqn:11}
ad=\Delta +bc,\qquad\text{with}\qquad 1\le  a,|b|,c,d \le  H.
\end{align}
We emphasise that $b$ need not be positive, but $b\neq 0$.

Suppose that the number of solutions of Equation~\eqref{eqn:11} for a fixed $a$ and $c$ is $G(a,c)$. We have \begin{align}\label{eqn:sum}
\#\cD_2^{1,1,1}(H,\Delta) = \sum_{0<a,c\le H} G(a,c).
\end{align} 

Our strategy of counting $\#\cD_2^{1,1,1}(H,\Delta)$ goes as follows: we first fix $c$, then count $\displaystyle \sum_{a\in I} G(a,c)$, the number of solutions of \eqref{eqn:11} as $a$ goes over an interval $I$. We then add this expression over some other interval $c\in J$, and finally add the remaining expression to get the number of solutions of~\eqref{eqn:11} over $0<a,c\le H$.

The problem of counting solutions to~\eqref{eqn:11} for a fixed $c$ is equivalent to counting the number of integer pairs $(a,d) \in [1,\dotsc,H]\times [1,\dotsc, H]$ that satisfy the equation \begin{align*}
ad \equiv \Delta \pmod{c},
\end{align*}
and also satisfy another inequality on $d$, which we will explain shortly.

From equation~\eqref{eqn:11}, we see that \begin{align*}
b=\dfrac{ad-\Delta}{c}.
\end{align*} Since $b\in [-H,H]$, we have \begin{align*}
-H\le  \dfrac{ad-\Delta}{c} \le  H \iff \dfrac{\Delta -Hc}{a}\le  d \le  \dfrac{\Delta + Hc}{a}.
\end{align*}

Hence, if $d$ satisfies Equation~\eqref{eqn:11}, \begin{align}\label{eqn:d'}
d \in [1,H]\cap \bigg[\dfrac{\Delta -Hc}{a}, \dfrac{\Delta + Hc}{a}\bigg].
\end{align}
With regards to this, we define some related functions, \begin{align}\label{eqn:func}
f_-(x) \colonequals \frac{\Delta-cH}{x} \qquad \text{ and }  \qquad		f_+(x) \colonequals \frac{\Delta+cH}{x}.
\end{align} 
Using~\eqref{eqn:func}, we rewrite~\eqref{eqn:d'} as
\begin{align}\label{eqn:d}
d \in [1,H]\cap [f_-(x),f_+(x)].
\end{align}

From \eqref{eqn:d'}, we see that the interval of $d$ depends on the sizes of $a$ and $c$. Suppose for some $a\in I_1$ and a fixed $c\in I_2$, where  $I_1,I_2 \subset (0,H]$, $d$ corresponds to an interval $I_3$, defined by~\eqref{eqn:d}. We then have \begin{align}\label{eqn:interval}
    \sum_{\substack{a\in I_1}} G(a,c) = \# \{a,d\in \Z^2 \colon a\in I_1,\: d\in I_3,\: ad\equiv \Delta\pmod{c}\}.
\end{align}

Depending on the intervals $I_1$ and $I_2$,~\eqref{eqn:interval} may correspond to a modular rectangle in modulo $c$, a modular hyperbola, or some union of them. We then apply Lemmas~\ref{lem:ModHyp-Box} or \ref{lem:Ustinov} to count~\eqref{eqn:interval}, and add the resulting expressions to Equation~\eqref{eqn:sum}.

Based on~\eqref{eqn:d} and the previous paragraph, we divide the counting based on the size of $a$ and $c$, relative to $\Delta$ and $H$. With regards to $a$, we divide the cases to \begin{align*}
0< a \le  \Delta/H +c,\qquad \text{and} \qquad  \Delta/H+c<a\le  H
\end{align*} We may call these cases ``small" $a$ and ``large" $a$, respectively. With regards to $c$, we divide the cases to the case of \begin{align*}
0<c\le  \Delta/H\qquad \text{and} \qquad \Delta/H<c\le  H.
\end{align*} We may call these cases ``small" $c$ and ``large" $c$, respectively. Since these two case divisions are independent, altogether we have four different cases to consider.

With respect to these intervals, we have \begin{align}\begin{split}\label{eqn:splitsum1}
					\#\cD_2^{1,1,1}(H,\Delta)&=\sum_{0<a,c\le H}G(a,c)\\
					&=  \sum_{\substack{0<a\le  c+\Delta/H\\0<c\le  \Delta/H}} G(a,c)+ \sum_{\substack{0<a\le  c+\Delta/H\\\Delta/H\le  c\le  H}}  G(a,c) \\&\quad +\sum_{\substack{\Delta/H+c< a\le   H\\0<c\le  \Delta/H}} G(a,c) + \sum_{\substack{\Delta/H+c< a\le   H\\\Delta/H\le  c\le  H}}  G(a,c).
				\end{split}
			\end{align}
We calculate the summations for each of these terms in the next sections.
			\subsection{Small \texorpdfstring{$a$}{a}, large \texorpdfstring{$c$}{c}}\label{salc} We first c $\displaystyle \sum_{\substack{0<a\le  c+\Delta/H\\\Delta/H<c\le  H}} G(a,c)$, the number of solutions of~\eqref{eqn:11} that also satisfy \begin{align}\label{eqn:intsalc}
    0<a\le  c+\Delta/H\qquad \text{ and } \qquad  \Delta/H<c\le  H.
\end{align} We first count $\displaystyle\sum_{0<a\le  c+\Delta/H} G(a,c)$ for a fixed $c\in (\Delta/H,H]$. From \eqref{eqn:intsalc}, the interval in~\eqref{eqn:d} is equivalent to $$d\in [1,H].$$ Thus, from~\eqref{eqn:interval}, we have \begin{align*}
    	&\sum_{0<a\le  c+\Delta/H} G(a,c) \\&= \# \{
				(a,d) \in \Z^2 \colon \Delta/H+c< a\le   H,\: 0<d \le   H,\:ad \equiv \Delta \pmod{c}
				\}.
\end{align*}

Applying Lemma~\ref{lem:ModHyp-Box}, we have that, for each $\Delta/H<c\le  H$,
\begin{align}\label{eqn:salc}\begin{split}
	\sum_{0<a\le  c+\Delta/H} G(a,c)&= N(\Delta,c;0,0,c+\Delta/H,H)
	\\&=\dfrac{H}{c} \sum_{r|\Delta}\sum_{\substack{0< a \le  c+\Delta/H\\\gcd(a,c)=r}} r + E_1'(c),
\end{split} 
\end{align}
where $E_1'(c)$ satisfies \begin{align*}
   |E_1'(c)|&\le c^{o(1)}\left[c^{1/2}+\bigg(c+\dfrac{\Delta}{H}\bigg)D_cc^{-1}+D_c\right] \\
   &\le H^{o(1)} [c^{1/2}+c^{-1}D_cH+D_c],
\end{align*} with $$D_c\colonequals\gcd(c,\Delta).$$
We note that the last inequality is true since $\Delta \le 2H^2.$

Now we can count the original summation  $\displaystyle \sum_{\substack{0<a\le  c+\Delta/H\\\Delta/H<c\le  H}} G(a,c)$. First, using Lemma~\ref{lem:gcdsum} for adding the error terms $E_1'(c)$, we have
\begin{align*}\bigg|\sum_{\Delta/H<c\le  H} E_1'(c) \bigg|&\le \sum_{0<c\le H } H^{o(1)}[c^{1/2}+c^{-1}D_cH+D_c]
 \\& \le H^{o(1)}(H^{3/2}+ H+H)
  \\& \le H^{3/2+o(1)}.
\end{align*}
For the main terms, from~\eqref{eqn:salc} we have \begin{align}\label{eqn:smallalargec}\begin{split}&\sum_{\substack{0<a\le  c+\Delta/H\\\Delta/H<c\le  H}} G(a,c)\\&\quad= \sum_{\Delta/H<c\le  H}\dfrac{H}{c} \sum_{r|\Delta}\sum_{\substack{0< a \le  c+\Delta/H\\\gcd(a,c)=r}} r +O(H^{3/2+o(1)})\\
&\quad=  H \sum_{r|\Delta}\sum_{\substack{ 0< a \le  c+\Delta/H \\ \Delta/H<c\le  H\\\gcd(a,c)=r}} \dfrac{r}{c} + O(H^{3/2+o(1)}),
\end{split} 
\end{align}which completes the counting of the solutions of~\eqref{eqn:11} that also satisfies~\eqref{eqn:intsalc}.

			\subsection{Small \texorpdfstring{$a$}{a}, small \texorpdfstring{$c$}{c}}\label{sasc}
Next, we count $\displaystyle \sum_{\substack{0<a\le  c+\Delta/H\\\Delta/H<c\le  H}} G(a,c)$, the number of solutions of~\eqref{eqn:11} that also satisfy \begin{align}\label{eqn:intsasc}
    0<a\le  c+\Delta/H\qquad \text{ and } \qquad  0<c \le \Delta/H.
\end{align} We first count $\displaystyle\sum_{0<a\le  c+\Delta/H} G(a,c)$ for a fixed $c\le \Delta/H$. From \eqref{eqn:intsasc}, the interval in~\eqref{eqn:d} is equivalent to   \begin{align*}
	d \in \bigg[\dfrac{\Delta-cH}{a},H\bigg] = [1,H]-(0,f_-(a)).
\end{align*} Thus, from~\eqref{eqn:interval}, we have \begin{align}\label{eqn:sasc-intro}\begin{split}
		&\sum_{0<a\le  c+\Delta/H} G(a,c)\\
		&\quad= 	\#\{(a,d) \in \Z^2 \colon  1\le  a \le  c+\Delta/H , 1\le  d \le  H,\\&\qquad ad \equiv \Delta \pmod{c} 
		\} \\ &\qquad- \#\{(a,d) \in \Z^2 \colon  1\le  a \le  c+\Delta/H , 0< d < f_-(a),\\&\qquad \quad ad \equiv \Delta \pmod{c} 
		\}.
	\end{split}
\end{align}
We note that \begin{align*}
    \#\{(a,d) \in \Z^2 \colon  1\le  a \le  c+\Delta/H , 1\le  d \le  H,\:ad \equiv \Delta \pmod{c} \}
\end{align*}
 is exactly \eqref{eqn:salc}. Therefore, to count \eqref{eqn:sasc-intro}, it only remains to calculate \begin{align*}\#\{(a,d) \in \Z^2 \colon  1\le  a \le  c+\Delta/H , 0< d < f_-(a),\:ad \equiv \Delta \pmod{c} .
\end{align*} To do this, we first recall the notation of modular hyperbola used in Lemma~\ref{lem:Ustinov}.  We have
\begin{align*}
	&\#\{(a,d) \in \Z^2 \colon  1\le  a \le  c+\Delta/H , 0< d < f_-(a),\:ad \equiv \Delta \pmod{c}  \\&\quad=  \#\{
	(a,d) \in \Z^2 \colon 0< a\le  c+\Delta/H,\:0<d\le   f_-(a),\:ad \equiv \Delta \pmod{c}
	\} \\
	&\quad=  T_{f_-}(\Delta,c;0,c+\Delta/H).
\end{align*}

In order to calculate $T_{f_-}(\Delta,c;0,c+\Delta/H)$, we first partition the interval $[1,c+\Delta/H)$ to $$I=\lfloor\log_2 (c+\Delta/H)\rfloor\asymp \log_2 (\Delta/H)$$ dyadic intervals of the form $[Z_i-U_i,Z_i) $, where \begin{align*}
	Z_i = 2^{1-i}(c+\Delta/H)\text{ and }U_i\le  Z_i,\:i=1,\dotsc,I.
\end{align*} In fact, we may choose $U_i=Z_i-Z_{i-1}$, except when $i=I$. Using this partition, we have 
\begin{align*}
	T_{f_-}(\Delta,c;0,c+\Delta/H)=\sum_{i=1}^I T_{f_-}(\Delta,c,Z_i-U_i,U_i).
\end{align*}
In each of the interval $[Z_i-U_i,Z_i)$, we have \begin{align*}
	|f_-''(x)| \asymp \dfrac{\Delta-cH}{Z_i^3} \asymp \dfrac{2^{3i}H^3(\Delta-cH)}{\Delta^3}.
\end{align*}

Applying Lemma~\ref{lem:Ustinov} to each of the term $T_{f_-}(\Delta,c;Z_i-U_i,U_i)$, we have
\begin{align}\label{eqn:sas2c}\begin{split}&T_{f_-}(\Delta,c;0,c+\Delta/H)\\&\quad=\sum_{i=1}^I T_{f_-}(\Delta,c;Z_i-U_i,U_i)\\
	&\quad=  \dfrac{1}{c}\sum_{r|\Delta}\sum_{r|\Delta}\sum_{\substack{0<a\le  c+\Delta/H\\\gcd(a,c)=r}} rf_-(a) -  \bigg(c+\dfrac{\Delta}{H}\bigg)\dfrac{\delta_c(\Delta)}{2}+E_{2,2}(c)\\
	&\quad=\dfrac{\Delta-cH}{c}\sum_{r|\Delta}\sum_{r|\Delta}\sum_{\substack{0< a \le  c+\Delta/H\\\gcd(a,c)=r}} \dfrac{r}{a} -  \bigg(c+\dfrac{\Delta}{H}\bigg)\dfrac{\delta_c(\Delta)}{2}+E_{2,2}(c),
\end{split}
\end{align} 
with
\begin{align*}
	|E_{2,2}(c)|\le & \sum_{i=1}^I c^{o(1)}\bigg[
	2^{1-i}\bigg(\dfrac{\Delta}{H}+c\bigg)\bigg(\dfrac{2^{3i}H^3(\Delta-cH)}{\Delta^3}\bigg)^{1/3}
	\\
	&\qquad +c^{-1}D_c^{1/2}\bigg(\dfrac{\Delta^3}{2^{3i}H^3(\Delta-cH)}\bigg)^{1/2}+ c^{1/2}+D_c
	\bigg]\\
	\le & H^{o(1)} [(\Delta-cH)^{1/3} + c^{-1}D_c^{1/2}  (\Delta-cH)^{-1/2}
	+c^{1/2}+D_c]\\
	\le & H^{o(1)} (\Delta^{1/3} +   c^{1/2}+D_c).
\end{align*}

		Substituting Equations~\eqref{eqn:salc} and~\eqref{eqn:sas2c} to~\eqref{eqn:sasc-intro}, we have that for all $0<c\le  \Delta/H$,
		\begin{align}\label{eqn:sasc}\begin{split}
			&\sum_{0<a\le  c+\Delta/H} G(a,c)\\ &\quad=\dfrac{H}{c} \sum_{r|\Delta}\sum_{\substack{0< a \le  c+\Delta/H\\\gcd(a,c)=r}} r - \dfrac{\Delta-cH}{c}\sum_{r|\Delta}\sum_{\substack{0<a\le  c+\Delta/H\\\gcd(a,c)=r}} \dfrac{r}{a} \\& \qquad+  \bigg(c+\dfrac{\Delta}{H}\bigg)\dfrac{\delta_c(\Delta)}{2}+E_2(c),
		\end{split}
		\end{align} where \begin{align*}
			|E_2(c)| \le H^{o(1)} (c^{1/2}+c^{-1}H+\Delta^{1/3} +D_c).
		\end{align*}
		
Now we can count the original summation $\displaystyle \sum_{\substack{0<a\le  c+\Delta/H\\0<c\le \Delta/H}} G(a,c)$. First, using Lemma~\ref{lem:gcdsum} for adding the error terms $E_2(c)$, we have
  \begin{align*}
      \bigg|\sum_{0<c\le \Delta/H} E_2(c)\bigg| &\le \sum_{0<c\le \Delta/H}  H^{o(1)} (c^{1/2}+c^{-1}H+\Delta^{1/3} +D_c)
      \\&\le H^{o(1)}[(\Delta/H)^{3/2}+H+\Delta^{1/3}(\Delta/H)+(\Delta/H)]
      \\& \le H^{o(1)}(H+\Delta^{4/3}H^{-1}),
  \end{align*} where the last inequality is true since $\Delta\le 2H^2$.
  
For the main terms, from~\eqref{eqn:sasc} we have 
  \begin{align*}
	&\sum_{\substack{0<a\le  c+\Delta/H\\0<c\le \Delta/H}} G(a,c)\\
 &\quad=\sum_{0<c\le \Delta/H} \bigg[ \dfrac{H}{c} \sum_{r|\Delta}\sum_{\substack{0< a \le  c+\Delta/H\\\gcd(a,c)=r}} r - \dfrac{\Delta-cH}{c}\sum_{r|\Delta}\sum_{r|\Delta}\sum_{\substack{0<a\le  c+\Delta/H\\\gcd(a,c)=r}} \dfrac{r}{a} \bigg]\\&\qquad + \sum_{0<c\le \Delta/H} \bigg(c+\dfrac{\Delta}{H}\bigg)\dfrac{\delta_c(\Delta)}{2}+O(H^{o(1)}[H+\Delta^{4/3}H^{-1}]) .
 \end{align*} Rearranging the summations and noting that \begin{align*}
     \sum_{0<c\le \Delta/H} \bigg(c+\dfrac{\Delta}{H}\bigg)\dfrac{\delta_c(\Delta)}{2} \le (\Delta/H)\Delta^{o(1)},
 \end{align*} we have
		\begin{align}\label{eqn:smallasmallc}
  \begin{split}
	\sum_{\substack{0<a\le  c+\Delta/H\\0<c\le \Delta/H}} G(a,c)&=  H \sum_{r|\Delta}\sum_{\substack{ 0< a \le  c+\Delta/H\\0<c\le \Delta/H\\\gcd(a,c)=r}} \bigg( \dfrac{r}{c}+\dfrac{r}{a} \bigg)\\&\quad - \Delta \sum_{r|\Delta}\sum_{\substack{ 0< a \le  c+\Delta/H\\0<c\le \Delta/H\\\gcd(a,c)=r}} \dfrac{r}{ac} \\&\quad +O(H^{o(1)}[H+\Delta^{4/3}H^{-1}]),
\end{split} 
\end{align}
which completes the counting of the solutions of~\eqref{eqn:11} that also satisfies~\eqref{eqn:intsasc}.

			\subsection{Large \texorpdfstring{$a$}{a}, large \texorpdfstring{$c$}{c}}\label{lalc} 
   Next, we count $\displaystyle \sum_{\substack{c+\Delta/H<a \le H\\\Delta/H<c\le  H}} G(a,c)$, the number of solutions of~\eqref{eqn:11} that also satisfy \begin{align}\label{eqn:intlalc}
    a> \Delta/H+c\qquad \text{ and } \qquad  c> \Delta/H.
\end{align}
We first count $\displaystyle\sum_{c+\Delta/H<a\le H} G(a,c)$ for a fixed $c$. From \eqref{eqn:intlalc}, the interval in~\eqref{eqn:d} is equivalent to 
\begin{align*}
				d \in \bigg[1,\dfrac{\Delta+Hc}{a}\bigg]=(0,f_+(a)].
			\end{align*} Thus, from~\eqref{eqn:interval}, we have
			\begin{align*}
				&\sum_{\Delta/H+c<a\le  H} G(a,c) \\&\quad= \# \{
				(a,d) \in \Z^2 \colon \Delta/H+c< a\le   H,\: 0<d \le   f_+(a),\:ad \equiv \Delta \pmod{c}
				\}
				\\&\quad= T_{f_+}(\Delta,c,\Delta/H+c,H-\Delta/H-c).
			\end{align*}
			
			We now partition the interval $(\Delta/H+c,H]$ to $$I=\bigg\lfloor \log_2 \dfrac{H}{\Delta/H+c}\bigg \rfloor \asymp \log_2 \dfrac{H^2}{\Delta+Hc}$$ dyadic intervals of the form $(Z_i,Z_i+U_i] $, where \begin{align*}
				Z_i = 2^{i-1}(\Delta/H+c)\text{ and }U_i\le  Z_i,\:i=1,\dotsc,I.
			\end{align*} Using this partition, we have
			\begin{align*}
				T_{f_+}(\Delta,c,\Delta/H+c,H-\Delta/H-c)=\sum_{i=1}^I T_{f_+}(\Delta,c,Z_i,U_i).
			\end{align*}
			In this interval, we have  \begin{align*}
				|f_+''(x)| \asymp \dfrac{\Delta+Hc}{Z_i^3} \asymp \dfrac{H^3}{2^{3i}(\Delta+Hc)^2}.
			\end{align*}
			
			Therefore, using Lemma~\ref{lem:Ustinov}, we have 
			\begin{align} \label{eqn:lalc}\begin{split}
				&T_{f_+}(\Delta,c,\Delta/H+c,H-\Delta/H-c)\\&\quad=\sum_{i=1}^I T_{f_+}(\Delta,c,Z_i,U_i)\\
				&\quad= \dfrac{1}{c}\sum_{r|\Delta}\sum_{\substack{\Delta/H+c< a\le   H\\\gcd(a,c)=r}} rf_+(a) -  \bigg(H-\dfrac{\Delta}{H}-c\bigg)\dfrac{\delta_c(\Delta)}{2}+E_3(c)
				\\&\quad= \dfrac{cH+\Delta}{c}\sum_{r|\Delta}\sum_{\substack{\Delta/H+c< a\le   H\\\gcd(a,c)=r}} \dfrac{r}{a} - \bigg(H-\dfrac{\Delta}{H}-c\bigg) \dfrac{\delta_c(\Delta)}{2}+E_3(c),
			\end{split}
			\end{align} where $E_3(c)$ satisfies 
			\begin{align*}\begin{split}
					|E_3(c)| &\le  \sum_{i=1}^IH^{o(1)}\bigg[2^{i-1}\bigg(c+\dfrac{\Delta}{H}\bigg)\bigg(\dfrac{H^3}{2^{3i}(\Delta+Hc)^2}\bigg)^{1/3}
					\\&\qquad +c^{-1}D_c^{1/2}\bigg(\dfrac{2^{3i}(\Delta+Hc)^2}{H^3}\bigg)^{1/2}+c^{1/2}+D_c\bigg]
					\\&\le H^{o(1)}[(\Delta+Hc)^{1/3}+c^{-1}D_c^{1/2}H^{3/2}(\Delta+Hc)^{-1/2}+c^{1/2}+D_c]
					\\&\le H^{o(1)}(c^{1/3}H^{1/3}+c^{-1}D_c^{1/2}H^{3/2}+c^{1/2}+D_c).
				\end{split} 
			\end{align*}
	
Now we can count the original summation $\displaystyle \sum_{\substack{c+\Delta/H<a\le H\\\Delta/H<c\le H}} G(a,c)$. First, using Lemma~\ref{lem:gcdsum} for adding the error terms $E_3(c)$, we have
\begin{align*}
       \bigg|\sum_{\Delta/H<c\le H} E_3(c) \bigg| &\le \sum_{0<c\le H}  H^{o(1)}[c^{1/3}H^{1/3}+c^{-1}D_c^{1/2}H^{3/2}+c^{1/2}+D_c]
     \\ &\le H^{o(1)}[\Delta^{1/3}H+H^{5/3}+H^{3/2}+H^{3/2}+H]
     \\& \le H^{5/3+o(1)}.
   \end{align*}
		
For the main terms, from~\eqref{eqn:lalc} we have
\begin{align}\label{eqn:largealargec}\begin{split}
	&\sum_{\substack{c+\Delta/H<a\le H\\\Delta/H<c\le H}} G(a,c)\\
			    &\quad= \sum_{\Delta/H<c\le H} \bigg[
       \dfrac{cH+\Delta}{c}\sum_{r|\Delta}\sum_{\substack{\Delta/H+c< a\le   H\\\gcd(a,c)=r}} \dfrac{r}{a} -   \\&\qquad \bigg(H-\dfrac{\Delta}{H}-c\bigg) \dfrac{\delta_c(\Delta)}{2}
       \bigg]+O(H^{5/3+o(1)})
       \\&\quad= H \sum_{r|\Delta}\sum_{\substack{\Delta/H<c\le H\\r|c}}\sum_{\substack{\Delta/H+c< a\le   H\\\gcd(a,c)=r}}\dfrac{r}{a}  \\&\qquad + \Delta \sum_{r|\Delta}\sum_{\substack{\Delta/H<c\le H\\r|c}}\sum_{\substack{\Delta/H+c< a\le   H\\\gcd(a,c)=r}}\dfrac{r}{ac}\\&\qquad +O(H\Delta^{o(1)})+O(H^{5/3+o(1)})
       \\&\quad= H \sum_{r|\Delta}\sum_{\substack{\Delta/H+c< a\le   H\\ \Delta/H<c\le H\\\gcd(a,c)=r}}\dfrac{r}{a} + \Delta \sum_{r|\Delta}\sum_{\substack{\Delta/H+c< a\le   H\\\Delta/H<c\le H\\\gcd(a,c)=r}}\dfrac{r}{ac} +O(H^{5/3+o(1)}),
			\end{split}\end{align}
which completes the counting of the solutions of~\eqref{eqn:11} that also satisfies~\eqref{eqn:intlalc}.
			
			\subsection{Large \texorpdfstring{$a$}{a}, small \texorpdfstring{$c$}{c}}\label{lasc}   
   Finally, we count $\displaystyle \sum_{\substack{c+\Delta/H<a \le H\\ 0<c\le \Delta/H}} G(a,c)$, the number of solutions of~\eqref{eqn:11} that also satisfy \begin{align}\label{eqn:intlasc}
    a> \Delta/H+c\qquad \text{ and } \qquad  0<c\le \Delta/H.
\end{align} We first count $\displaystyle\sum_{c+\Delta/H<a\le H} G(a,c)$ for a fixed $c$. From \eqref{eqn:intlasc}, the interval in~\eqref{eqn:d} is equivalent to 
\begin{align*}
				d \in \bigg[\dfrac{\Delta-Hc}{a},\dfrac{\Delta+Hc}{a}\bigg] =(0,f_+(a)]-(0,f_-(a)].
			\end{align*} Thus, from~\eqref{eqn:interval}, we have
			\begin{align}\label{eqn:lasc-intro}\begin{split}
				&\sum_{\Delta/H+c<a\le  H} G(a,c) \\
				&\quad= \# \{
				(a,d) \in \Z^2 \colon \Delta/H+c< a\le   H,\: 0<d \le   f_+(a),\\&\qquad \quad ad \equiv \Delta \pmod{c}
				\}\\
				&\qquad - \# \{
				(a,d) \in \Z^2 \colon \Delta/H+c< a\le   H,\: 0<d \le   f_-(a),\\& \qquad \quad ad \equiv \Delta \pmod{c}
				\}\\
				&\quad= T_{f_+}(\Delta,c,\Delta/H+c,H-(\Delta/H+c))\\&\qquad -T_{f_-}(\Delta,c,\Delta/H+c,H-(\Delta/H+c)).\end{split}\end{align}
			
			We note that  $$T_{f_+}(\Delta,c,\Delta/H+c,H-(\Delta/H+c))$$ is exactly \eqref{eqn:lalc}. Therefore, to count \eqref{eqn:lasc-intro}, it only remains to calculate $$T_{f_-}(\Delta,c,\Delta/H+c,H-(\Delta/H+c)).$$ As in the previous section, we partition the interval $(\Delta/H+c,H]$ to $I$ dyadic intervals of the form $(Z_i,Z_i+U_i]$. We have
			\begin{align*}
				T_{f_-}(\Delta,c,\Delta/H+c,H-\Delta/H-c)=\sum_{i=1}^I T_{f_-}(\Delta,c,Z_i,U_i),
			\end{align*}
			and  \begin{align*}
				|f_-''(x)| \asymp \dfrac{\Delta-Hc}{Z_i^3} \asymp \dfrac{H^3(\Delta-Hc)}{2^{3i}\Delta^3}.
			\end{align*}
			
			Letting  $h'=H-(\Delta/H+c)$ and applying Lemma~\ref{lem:Ustinov}, we have
			\begin{align}\label{eqn:Tf-}\begin{split}
				&T_{f_-}(\Delta,c,\Delta/H+c,H-\Delta/H-c)\\
				&\quad = \dfrac{1}{c}\sum_{r|\Delta}\sum_{\substack{\Delta/H+c< a\le   H\\\gcd(a,c)=r}} rf_-(a) -  \dfrac{h'\delta_c(\Delta)}{2}+E_{4}^-(c),
			\end{split}
			\end{align}
			where $E_4^-(c)$ satisfies
			\begin{align*}
				|E_4^-(c)|  &\le \sum_{i=1}^I H^{o(1)}\bigg[2^{i-1}\bigg(c+\dfrac{\Delta}{H}\bigg)\bigg(\dfrac{H^3(\Delta-Hc)}{2^{3i}\Delta^3}\bigg)^{1/3}
				\\&\qquad +c^{-1}D_c^{1/2}\bigg(\dfrac{2^{3i}\Delta ^3}{H^3(\Delta-Hc)}\bigg)^{1/2}+c^{1/2}+D_c\bigg]
				\\&\le H^{o(1)} [(\Delta-Hc)^{1/3}
			+ c^{-1}D_c^{1/2}H^{3/2}(\Delta-Hc)^{-1/2}+c^{1/2}+D_c]
				\\&\le H^{o(1)}[\Delta^{1/3}+c^{-1}D_c^{1/2}H^{3/2}+c^{1/2}+D_c].
			\end{align*}
   
			Substituting~\eqref{eqn:lalc}~and~\eqref{eqn:Tf-} to~\eqref{eqn:lasc-intro}, we have
	\begin{align}\label{eqn:lasc}\begin{split}
				&\sum_{\Delta/H+c<a\le  H} G(a,c)\\&\quad =\dfrac{1}{c} \sum_{r|\Delta}\sum_{\substack{\Delta/H+c< a\le   H\\\gcd(a,c)=r}} r \bigg[ \dfrac{(cH+\Delta)-(-Hc+\Delta)}{a} \bigg] +E_4(c)\\
				&\quad = 2H \sum_{r|\Delta}\sum_{\substack{\Delta/H+c< a\le   H\\\gcd(a,c)=r}} \dfrac{r}{a}+E_4(c),
			\end{split}
			\end{align}
   where $E_4(c)$ satisfies \begin{align*}
				|E_4(c)| &\le H^{o(1)} 
(\Delta^{1/3}+c^{1/3}H^{1/3}+c^{-1}D_c^{1/2}H^{3/2}+c^{1/2}+D_c
\\&\qquad +\Delta^{1/3}+c^{-1}D_c^{1/2}H^{3/2}+c^{1/2}+D_c)
				\\ & \le H^{o(1)}(\Delta^{1/3}+c^{1/3}H^{1/3}+c^{-1}D_c^{1/2}H^{3/2}+c^{1/2}+D_c).
			\end{align*}

Now we can count the original summation $\displaystyle \sum_{\substack{c+\Delta/H <a\le H\\0<c\le \Delta/H}} G(a,c)$. First, using Lemma~\ref{lem:gcdsum} for adding the error terms $E_4(c)$, we have
   \begin{align*}
      \bigg |\sum_{0<c\le \Delta/H } E_4(c) \bigg| &\le \sum_{0<c\le \Delta/H } H^{o(1)}(\Delta^{1/3}+c^{1/3}H^{1/3}+c^{-1}D_c^{1/2}H^{3/2}+c^{1/2}+D_c)
       \\ & \le H^{o(1)}(\Delta^{1/3}H+H^{5/3}+H^{3/2}+H^{3/2}+H)
       \\ & \le H^{5/3+o(1)}.
   \end{align*}
   
For the main terms, from~\eqref{eqn:lasc} we have\begin{align}\label{eqn:largeasmallc}
   \begin{split}
	&\sum_{\substack{c+\Delta/H <a\le H\\0<c\le \Delta/H}} G(a,c)
 \\ &\quad =  \sum_{0<c\le \Delta/H} \bigg[ 2H \sum_{r|\Delta}\sum_{\substack{\Delta/H+c< a\le   H\\\gcd(a,c)=r}} \dfrac{r}{a}+ E_4(c)\bigg ]
 \\&\quad =\:2H \sum_{r|\Delta}   \sum_{\substack{\Delta/H+c< a\le   H\\0<c\le \Delta/H\\\gcd(a,c)=r}} \dfrac{r}{a} + O(H^{5/3+o(1)}),
 \end{split}\end{align}
which completes the counting of the solutions of~\eqref{eqn:11} that also satisfies~\eqref{eqn:intlasc}.

			\subsection{Conclusion} We are now ready to prove Lemma~\ref{lem:sign} by counting $\#\cD_2^{1,1,1}(H,\Delta)$. for $0<\Delta<H^2$.
   First, from~\eqref{eqn:splitsum1}, \begin{align}\label{eqn:sumgac}\begin{split}
					\#\cD_2^{1,1,1}(H,\Delta)&=\sum_{\substack{0<a\le  c+\Delta/H\\0<c\le  \Delta/H}} G(a,c)+ \sum_{\substack{0<a\le  c+\Delta/H\\\Delta/H\le  c\le  H}}  G(a,c) \\&\quad +\sum_{\substack{\Delta/H+c< a\le   H\\0<c\le  \Delta/H}} G(a,c) + \sum_{\substack{\Delta/H+c< a\le   H\\\Delta/H\le  c\le  H}}  G(a,c).
				\end{split}
			\end{align}  Denote $\Delta=H^{1+\alpha}$ for some real $\alpha$. 
We recall our results in Sections~\ref{salc},~\ref{sasc},~\ref{lalc} and~\ref{lasc}, in particular Equations~\eqref{eqn:smallalargec},~\eqref{eqn:smallasmallc},~\eqref{eqn:largealargec} and~\eqref{eqn:largeasmallc}. Substituting these to~\eqref{eqn:sumgac}, we have 
   \begin{align}\begin{split}\label{eqn:medium1}
       &\#\cD_2^{1,1,1}(H,\Delta)\\
       &\quad= H \sum_{r|\Delta}\sum_{\substack{ 0< a \le  c+\Delta/H\\0<c\le \Delta/H\\\gcd(a,c)=r}} \bigg( \dfrac{r}{c}+\dfrac{r}{a} \bigg) - \Delta \sum_{r|\Delta}\sum_{\substack{ 0< a \le  c+\Delta/H\\0<c\le \Delta/H\\\gcd(a,c)=r}} \dfrac{r}{ac} 
       \\&\qquad +H \sum_{r|\Delta}\sum_{\substack{ 0< a \le  c+\Delta/H \\ \Delta/H<c\le  H\\\gcd(a,c)=r}} \dfrac{r}{c} + 2H \sum_{r|\Delta}   \sum_{\substack{\Delta/H+c< a\le   H\\0<c\le \Delta/H\\\gcd(a,c)=r}} \dfrac{r}{a} 
       \\&\qquad +H \sum_{r|\Delta}\sum_{\substack{\Delta/H+c< a\le   H\\ \Delta/H<c\le H\\\gcd(a,c)=r}}\dfrac{r}{a} + \Delta \sum_{r|\Delta}\sum_{\substack{\Delta/H+c< a\le   H\\\Delta/H<c\le H\\\gcd(a,c)=r}}\dfrac{r}{ac} \\&\qquad + O(H^{5/3+o(1)}).
       \end{split}\end{align}
       Rearranging the summations in~\eqref{eqn:medium1} and using Lemma~\ref{lem:xy}, we have \begin{align*}
       \begin{split} &\#\cD_2^{1,1,1}(H,\Delta)
        \\&\quad = H \sum_{r|\Delta}\sum_{\substack{ 0< a \le  c+\Delta/H \\ 0<c\le  H\\\gcd(a,c)=r}} \dfrac{r}{c}+  H \sum_{r|\Delta}\sum_{\substack{\Delta/H+c< a\le   H\\0<c\le H\\\gcd(a,c)=r}} \dfrac{r}{a} + H \sum_{r|\Delta}   \sum_{\substack{ 0< a\le   H\\0<c\le \Delta/H\\\gcd(a,c)=r}} \dfrac{r}{a} 
        \\&\qquad +O\bigg( \Delta \sum_{r|\Delta}\sum_{\substack{0< a,c \le   H\\\gcd(a,c)=r}}\dfrac{r}{ac} \bigg)   +O(H^{5/3+o(1)})
        \\&\quad = H\sum_{\substack{r|\Delta\\r\le H}} \left[\dfrac{6}{\pi^2}
        \dfrac{H}{r} + O\left(\dfrac{\Delta^{1+o(1)}}{H}\right)
        \right] + H\sum_{\substack{r|\Delta\\r\le H}} \left[\dfrac{6}{\pi^2}
        \dfrac{H}{r} + O\left(\dfrac{\Delta^{1+o(1)}}{H}\right)
        \right]
        \\&\qquad + H\sum_{\substack{r|\Delta\\r\le H}} \left[\dfrac{6}{\pi^2}\dfrac{\Delta}{Hr} +O\left(\dfrac{\Delta^{1+o(1)}}{H}\right) \right]  +O(\Delta^{1+o(1)})+O(H^{5/3+o(1)})
        \\&\quad = \dfrac{12}{\pi^2} H^2 \sum_{\substack{r|\Delta\\r\le H}} \dfrac{1}{r} +O(H^{o(1)}\max(H^{5/3},\Delta)).
   \end{split}
   \end{align*}
   This completes the proof of Lemma~\ref{lem:sign} for $\#\cD_2^{1,1,1}(H,\Delta)$.

			\section{Counting \texorpdfstring{$\#\cD_2^{1,1,-1}(H,\Delta)$}{D211-1(H,Delta)}}\label{sec:11-1}
\subsection{The case divisions}			We now count $\#\cD_2^{1,1,-1}(H,\Delta)$, which is the number of integer solutions to the equation \begin{align*}
				ad=\Delta +bc,\qquad\text{with}\qquad 1\le  a,|b|,c,-d \le  H.
			\end{align*}
			After some  variable changes, we may instead consider the equivalent problem of counting the number of integer solutions to the equation \begin{align}\label{eqn:1-1}
				ad=\Delta +bc,\qquad\text{with}\qquad 1\le  a,b,c,d \le  H.
			\end{align} 
			
			We first note that the equation~\eqref{eqn:1-1} does not have any solutions for $\Delta \ge H^2$. Therefore, we can assume $\Delta <H^2$. 

   Also, from  this observation, we have  \begin{align*}
     \#\cD_2^{1,1,-1}(H,\Delta)=  \sum_{n=1}^{N^2}  \tau_N(n)\tau_N(n+\Delta), 
   \end{align*} where $\tau_N$ is defined in~\eqref{eqn:tau}. Therefore, counting $\#\cD_2^{1,1,-1}(H,\Delta)$ also proves Corollary~\ref{cor:tau}.
   
			Suppose that the number of solutions of Equation~\eqref{eqn:1-1} for a fixed $a$ and $c$ is $J(a,c)$. We have \begin{align*}
				\#\cD_2^{1,1,-1}(H,\Delta) = \sum_{0<c\le H}\sum_{0<a\le H} J(a,c).
			\end{align*} 
			
			Similarly as in Section~\ref{sec:casediv}, we count $\sum J(a,c)$, the number of solutions  \eqref{eqn:1-1} for a fixed $c$ and $a$ in an interval, and add the terms over $0<c\le H$. The problem of counting solutions to~\eqref{eqn:1-1} for a fixed $c$ is equivalent to counting the number of tuples $(a,d) \in [1,\dotsc,H]\times [1,\dotsc, H]$ that satisfies the equation \begin{align*}
				ad \equiv \Delta \pmod{c},
			\end{align*} and also satisfies another inequality with respect to $d$, which we will explain shortly.
			
			For all solutions of \eqref{eqn:1-1}, we have that \begin{align*}
				0< b=\dfrac{ad-\Delta}{c} \le  H \iff \dfrac{\Delta}{a} < d \le  \dfrac{\Delta+cH}{a}.
			\end{align*}
			Therefore, we have \begin{align}\label{eqn:d2}
				d \in (0,H]\cap \bigg(\dfrac{\Delta}{a}, \dfrac{\Delta + Hc}{a}\bigg].
			\end{align}
			With regards to these intervals, we define \begin{align*}
				f_*(x)\colonequals \dfrac{\Delta}{x}.
			\end{align*}
			
   Using similar arguments as in Section~\ref{sec:casediv}, we have that for some $a\in I_1$ and a fixed $c\in I_2$, where  $I_1,I_2 \subset (0,H]$, $d$ corresponds to an interval $I_3$, defined by~\eqref{eqn:d2}. We then have \begin{align}\label{eqn:intervalg}
    \sum_{\substack{a\in I_1}} J(a,c) = \# \{a,d\in \Z^2 \colon a\in I_1,\: d\in I_3,\:ad\equiv \Delta\pmod{c}\}.
\end{align} 

			Next, similar as in Section~\ref{sec:casediv}, we divide our counting on $\sum J(a,c)$ based on the size of $a$ relative to $c$ and $\Delta$. In particular, we consider the cases \begin{align*}
				0<a\le \Delta/H+c \qquad \text{and} \qquad \Delta/H+c<a\le H.
			\end{align*}
			
			Let these cases be the case of \textit{small} $a$ and \textit{large} $a$ respectively. We have \begin{align*}
				&\#\cD_2^{1,1,-1}(H,\Delta)
				\\&=
				\sum_{\substack{0<a\le \Delta/H+c\\0<c\le H}} J(a,c)+\sum_{\substack{\Delta/H+c<a\le H\\0<c\le H}} J(a,c).
			\end{align*}

			\subsection{The case of small \texorpdfstring{$a$}{a}}
   First, we count $\displaystyle \sum_{\substack{0<a\le \Delta/H+c\\0<c\le H}}  J(a,c)$, the number of solutions of~\eqref{eqn:1-1} that also satisfies \begin{align}\label{eqn:intsa}
       0 < a\le   \Delta/H +c.
   \end{align}
			We first count $\displaystyle\sum_{0<a\le \Delta/H+c} J(a,c)$ for a fixed $c$. From \eqref{eqn:intsa}, the interval in~\eqref{eqn:d2} is equivalent to \begin{align*}
			d\in \bigg( \dfrac{\Delta}{a},H \bigg] = (0,H]-(0,f_*(a)].
			\end{align*}
			Thus, from~\eqref{eqn:intervalg}, we have \begin{align}\label{eqn:sa-intro}\begin{split}
				&\sum_{0<a\le \Delta/H+c} J(a,c)
				\\&\quad =\# \{
				(a,d) \in \Z^2 \colon 0< a\le   \Delta/H+c ,\: 0<d \le   H,\\&\qquad \quad  ad \equiv \Delta \pmod{c}
				\}
				\\ &\qquad - \# \{
				(a,d) \in \Z^2 \colon 0< a\le   \Delta/H+c ,\: 0<d \le   f_*(a),\\&\qquad \quad  ad \equiv \Delta \pmod{c}
				\}.
			\end{split}
			\end{align}
			
			To calculate~\eqref{eqn:sa-intro}, we first recall our results from Section~\ref{salc}, and in particular Equation~\eqref{eqn:salc}. We  have \begin{align}\label{eqn:sa1}\begin{split}
				&\# \{
				(a,d) \in \Z^2 \colon 0< a\le   \Delta/H+c ,\: 0<d \le   H,\\&\quad ad \equiv \Delta \pmod{c}
				\}\\
    &\quad = \dfrac{H}{c} \sum_{r|\Delta}\sum_{\substack{0< a \le  c+\Delta/H\\\gcd(a,c)=r}} r + O(H^{o(1)} [c^{1/2}+c^{-1}H+D_c]).
			    \end{split}
			\end{align}
			
			Next, we have \begin{align*}
				&\# \{
				(a,d) \in \Z^2 \colon 0< a\le   \Delta/H+c ,\: 0<d \le   f_*(a),\:ad \equiv \Delta \pmod{c}
				\}.
				\\
				&\quad =T_{f_*} (\Delta,c;0,c+\Delta/H).
			\end{align*}

			Similarly as Section~\ref{sasc}, we partition the interval $[1,c+\Delta/H)$ to $$I=\lfloor\log (c+\Delta/H)\rfloor\asymp \log \dfrac{\Delta+cH}{H}$$ dyadic intervals of the form $[Z_i-U_i,Z_i) $, where \begin{align*}
				Z_i = 2^{1-i}(c+\Delta/H)\text{ and }U_i\le  Z_i,\:i=1,\dotsc,I.
			\end{align*} Using this partition, we have 
			\begin{align*}
				T_{f_*}(\Delta,c;0,c+\Delta/H)=\sum_{i=1}^I T_{f_*}(\Delta,c,Z_i-U_i,U_i).
			\end{align*}
			In each of the interval $[Z_i-U_i,Z_i)$, we have \begin{align*}
				|f_*''(x)| \asymp \dfrac{\Delta}{Z_i^3} \asymp \dfrac{2^{3i}H^3\Delta}{(\Delta+cH)^3}.
			\end{align*}
			Thus, using Lemma~\ref{lem:Ustinov}, we have
\begin{align}\label{eqn:sa2}\begin{split}&T_{f_*}(\Delta,c;0,c+\Delta/H)\\&\quad =\sum_{i=1}^I T_{f_*}(\Delta,c;Z_i-U_i,U_i)\\
				\\&\quad =  \dfrac{1}{c}\sum_{r|\Delta}\sum_{r|\Delta}\sum_{\substack{0<a\le  c+\Delta/H\\\gcd(a,c)=r}} rf_*(a) -  \bigg(c+\dfrac{\Delta}{H}\bigg)\dfrac{\delta_c(\Delta)}{2}+E_5'(c)\\
				&\quad =\dfrac{\Delta}{c}\sum_{r|\Delta}\sum_{\substack{0< a \le  c+\Delta/H\\\gcd(a,c)=r}} \dfrac{r}{a} -  \bigg(c+\dfrac{\Delta}{H}\bigg)\dfrac{\delta_c(\Delta)}{2}+E_5'(c),
			    \end{split}
			\end{align} 
			with $E_5(c)$ satisfying
			\begin{align*}
				|E_5(c)|\le & \sum_{i=1}^I c^{o(1)}\bigg[
				2^{1-i}\bigg(\dfrac{\Delta}{H}+c\bigg)\bigg(\dfrac{2^{3i}H^3\Delta}{(\Delta+cH)^3}\bigg)^{1/3}
				\\
				&\qquad +c^{-1}D_c^{1/2}\bigg(\dfrac{(\Delta+cH)^3}{2^{3i}H^3\Delta}\bigg)^{1/2}+ c^{1/2}+D_c
				\bigg]\\
				\le & H^{o(1)} (\Delta^{1/3} + c^{-1}D_c^{1/2}  \Delta^{-1/2}
				+c^{1/2}+D_c).\\
				\le & H^{o(1)} (\Delta^{1/3} + c^{-1}D_c^{1/2}+  c^{1/2}+D_c).
			\end{align*}
			
			Substituting~\eqref{eqn:sa1}~and~\eqref{eqn:sa2} to~\eqref{eqn:sa-intro}, we have that for all $0<c\le  \Delta/H$.
			\begin{align}\label{eqn:sjac}\begin{split}
					\sum_{0<a\le \Delta/H+c} J(a,c)&=\dfrac{H}{c} \sum_{r|\Delta}\sum_{\substack{0< a \le  c+\Delta/H\\\gcd(a,c)=r}} r  -
					\dfrac{\Delta}{c}\sum_{r|\Delta}\sum_{r|\Delta}\sum_{\substack{0<a\le  c+\Delta/H\\\gcd(a,c)=r}} \dfrac{r}{a} \\&\quad +  \bigg(c+\dfrac{\Delta}{H}\bigg)\dfrac{\delta_c(\Delta)}{2}+E_5(c),
				\end{split}
			\end{align} with \begin{align*}
				|E_5(c)| \le H^{o(1)} [c^{1/2}+c^{-1}H+c^{-1}D_cH^{-1}\Delta+\Delta^{1/3} + c^{-1}D_c^{1/2}+D_c].
			\end{align*}
			Now we can count the original summation $\displaystyle \sum_{\substack{0<a\le \Delta/H+c\\0<c\le H}}  J(a,c) $. First, using Lemma~\ref{lem:gcdsum} for adding the error terms $E_5(c)$, we have
			\begin{align*}
				\begin{split}
					\bigg|\sum_{0<c\le H } E_5(c)\bigg| &\le \sum_{0<c\le H} H^{o(1)} [c^{1/2}+c^{-1}H+c^{-1}D_cH^{-1}\Delta+\Delta^{1/3}\\& + c^{-1}D_c^{1/2}+D_c]\\
					& \le  H^{o(1)} (H^{3/2}+ H^{-1}\Delta+H\Delta^{1/3}).
				\end{split}
			\end{align*}
   
For the main terms, from~\eqref{eqn:sjac} we have
			\begin{align}\label{eqn:sac1}
				\begin{split}
					&
				\sum_{\substack{0<a\le \Delta/H+c\\0<c\le H}}  J(a,c) 
    \\&\quad =\sum_{0<c\le H }\bigg[ \dfrac{H}{c} \sum_{r|\Delta}\sum_{\substack{0< a \le  c+\Delta/H\\\gcd(a,c)=r}} r  \bigg] - \sum_{0<c\le H }\bigg[\dfrac{\Delta}{c}\sum_{r|\Delta}\sum_{\substack{0<a\le  c+\Delta/H\\\gcd(a,c)=r}} \dfrac{r}{a} \bigg] 
    \\&\qquad +  \sum_{0<c\le H }\bigg(c+\dfrac{\Delta}{H}\bigg)\dfrac{\delta_c(\Delta)}{2} + \sum_{0<c\le H } E_5(c)
					\\
					&\quad = H\sum_{r|\Delta}\sum_{\substack{0<a\le  c+\Delta/H\\0<c\le H\\\gcd(a,c)=r}} \dfrac{r}{c} - \Delta \sum_{r|\Delta}\sum_{\substack{0<a\le  c+\Delta/H\\0<c\le H\\\gcd(a,c)=r}} \dfrac{r}{ac}\\
					&\qquad + O(H^{3/2+o(1)}+H^{1+o(1)}\Delta^{1/3}+\Delta^{1+o(1)}).
				\end{split}
			\end{align}

   We can further simplify~\eqref{eqn:sac1}. First, we note that $H^{1+o(1)}\Delta^{1/3}$ is dominated by $H^{3/2+o(1)}$ if $\Delta \le H^{3/2}$ and by $\Delta^{1+o(1)}$ if $\Delta >H^{3/2}$. Thus, we can remove this expression from the error terms in~\eqref{eqn:sac1}.
   
   Next,  by rearranging terms and applying Lemma~\ref{lem:xy}, we have
   \begin{align}\label{eqn:sac}
       \begin{split}
           &\sum_{\substack{0<a\le \Delta/H+c\\0<c\le H}}  J(a,c) 
 \\   &\quad = H\sum_{r|\Delta}\sum_{\substack{0<a\le  c+\Delta/H\\0<c\le H\\\gcd(a,c)=r}} \dfrac{r}{c} +O\left( \Delta \sum_{r|\Delta}\sum_{\substack{0<a,c\le H\\\gcd(a,c)=r}} \dfrac{r}{ac}\right)
 \\&\qquad + O(H^{o(1)}\max(H^{3/2},\Delta))
 \\ &\quad = H\sum_{\substack{r|\Delta\\r\le H}} \left[\dfrac{6}{\pi^2} \dfrac{H}{r}+ O\left(\dfrac{\Delta^{1+o(1)}}{Hr}\right)\right] 
 + O(H^{o(1)}\max(H^{3/2},\Delta))
 \\ &\quad = \dfrac{6}{\pi^2} H^2 \sum_{\substack{r|\Delta\\r\le H}}\dfrac{1}{r}+ O(H^{o(1)}\max(H^{3/2},\Delta)),
       \end{split}
   \end{align}
			which completes the counting of the solutions of~\eqref{eqn:1-1} that also satisfies~\eqref{eqn:intsa}.

			\subsection{The case of large \texorpdfstring{$a$}{a}}
Next, we count $\displaystyle \sum_{\substack{\Delta/H+c <a \le  H \le \Delta+c \\0<c\le H}}  J(a,c)$, the number of solutions of~\eqref{eqn:1-1} that also satisfies \begin{align}\label{eqn:intla}
      \Delta/H +c <a \le  H.
   \end{align}
			We first count $\displaystyle\sum_{\Delta/H+c <a \le  H} J(a,c)$ for a fixed $c$. From \eqref{eqn:intla}, the interval in~\eqref{eqn:d2} is equivalent to \begin{align*}
			d\in \bigg(\dfrac{\Delta}{a} ,\dfrac{\Delta+Hc}{a}\bigg] = (0,f_+(a)]-(0,f_*(a)].
			\end{align*}
			Thus, from~\eqref{eqn:intervalg}, we have 
			\begin{align}\label{eqn:la-intro}
    \begin{split}
				&\sum_{\Delta/H +c  \le  a\le  H} J(a,c) \\&\quad = \# \{
				(a,d) \in \Z^2 \colon \Delta/H+c <  a\le   H ,\: 0<d \le   f_+(a),\\&\qquad ad \equiv \Delta \pmod{c}
				\}
    \\ &\qquad -\# \{
				(a,d) \in \Z^2 \colon \Delta/H+c <  a\le   H,\: 0<d \le   f_*(a),\\&\qquad ad \equiv \Delta \pmod{c}
				\}
				\\&\quad = T_{f_+}(\Delta,c,\Delta/H +c,H-\Delta/H -c)\\&\qquad-T_{f_*} (\Delta,c,\Delta/H +c,H-\Delta/H -c).
    \end{split}
			\end{align}
			We note that \begin{align*}
       T_{f_+} (\Delta,c,\Delta/H +c,H-\Delta/H -c).
   \end{align*}
   is exactly~\eqref{eqn:lalc}. Therefore, to count~\eqref{eqn:la-intro}, it only remains to calculate\begin{align*}
       T_{f_*} (\Delta,c,\Delta/H +c,H-\Delta/H -c).
   \end{align*}
   
			As in Section~\ref{lalc}, we partition the interval $(\Delta/H+c,H]$ to $$I=\bigg\lfloor \log_2 \dfrac{H}{\Delta/H+c}\bigg \rfloor \asymp \log_2 \dfrac{H^2}{\Delta+Hc}$$ dyadic intervals of the form $(Z_i,Z_i+U_i] $, where \begin{align*}
				Z_i = 2^{i-1}(\Delta/H+c)\text{ and }U_i\le  Z_i,\:i=1,\dotsc,I.
			\end{align*} 
			In this interval, we have  \begin{align*}
				|f_*''(x)| \asymp \dfrac{\Delta}{Z_i^3} \asymp \dfrac{H^3\Delta}{2^{3i}(\Delta+Hc)^2}.
			\end{align*}
   
		Therefore, using Lemma~\ref{lem:Ustinov}, we have 
			\begin{align} \label{eqn:ljac2}\begin{split}
				&T_{f_*}(\Delta,c,\Delta/H+c,H-\Delta/H-c)\\&\quad =\sum_{i=1}^I T_{f_*}(\Delta,c,Z_i,U_i)\\
				&\quad = \dfrac{1}{c}\sum_{r|\Delta}\sum_{\substack{\Delta/H+c< a\le   H\\\gcd(a,c)=r}} rf_*(a) -  \bigg(H-\dfrac{\Delta}{H}-c\bigg)\dfrac{\delta_c(\Delta)}{2}+E_6^*(c)
				\\&\quad = \dfrac{\Delta}{c}\sum_{r|\Delta}\sum_{\substack{\Delta/H+c< a\le   H\\\gcd(a,c)=r}} \dfrac{r}{a} - \bigg(H-\dfrac{\Delta}{H}-c\bigg) \dfrac{\delta_c(\Delta)}{2}+E_6^*(c),
			\end{split}
			\end{align} where $E_6^*(c)$ satisfies 
	\begin{align*}\begin{split}
					|E_6^*(c)| \le  &\sum_{i=1}^IH^{o(1)}\bigg[2^{i-1}\bigg(c+\dfrac{\Delta}{H}\bigg)\bigg(\dfrac{H^3\Delta}{2^{3i}(\Delta+Hc)^3}\bigg)^{1/3}
					\\&\qquad +c^{-1}D_c^{1/2}\bigg(\dfrac{2^{3i}(\Delta+Hc)^3}{H^3\Delta}\bigg)^{1/2}+c^{1/2}+D_c\bigg]
					\\&\le H^{o(1)}(\Delta^{1/3}+c^{-1}D_c^{1/2}H^{3/2}\Delta^{-1/2}+c^{1/2}+D_c)
					\\&\le H^{o(1)}(c^{1/3}H^{1/3}+c^{-1}D_c^{1/2}H^{3/2}+c^{1/2}+D_c).
				\end{split} 
			\end{align*}

   Substituting \eqref{eqn:lalc} and \eqref{eqn:ljac2} to \eqref{eqn:la-intro}, we have
   \begin{align}\label{eqn:ljac}
       \begin{split}
           &\sum_{\Delta/H +c  \le  a\le  H} J(a,c)
           \\&\quad = \dfrac{cH+\Delta}{c}\sum_{r|\Delta}\sum_{\substack{\Delta/H+c< a\le   H\\\gcd(a,c)=r}} \dfrac{r}{a} - \bigg(H-\dfrac{\Delta}{H}-c\bigg) \dfrac{\delta_c(\Delta)}{2}
           \\&\qquad - \dfrac{\Delta}{c}\sum_{r|\Delta}\sum_{\substack{\Delta/H+c< a\le   H\\\gcd(a,c)=r}} \dfrac{r}{a} + \bigg(H-\dfrac{\Delta}{H}-c\bigg) \dfrac{\delta_c(\Delta)}{2}+E_6(c)
           \\&\quad = H \sum_{r|\Delta}\sum_{\substack{\Delta/H+c< a\le   H\\\gcd(a,c)=r}} \dfrac{r}{a}+E_6(c),
       \end{split}
   \end{align}
   where $E_6(c)$ satisfies \begin{align*}
       |E_6(c)| \le H^{o(1)}(c^{1/3}H^{1/3}+c^{-1}D_c^{1/2}H^{3/2}+c^{1/2}+D_c).
   \end{align*}
				Now we can count the original summation $\displaystyle\sum_{\substack{\Delta/H+c<a\le H\\0<c\le H}} J(a,c) $. First, using Lemma~\ref{lem:gcdsum} for adding the error terms $E_6(c)$, we have 
			\begin{align*}
				\left|\sum_{0<c\le H} E_6(c)\right| &\le \sum_{0<c\le H}  H^{o(1)}(c^{1/3}H^{1/3}+c^{-1}D_c^{1/2}H^{3/2}+c^{1/2}+D_c)\\
    &\le H^{o(1)} (H^{5/3}+H^{3/2}+H^{3/2}+H) \\&\le H^{5/3+o(1)}.
			\end{align*}
			
For the main terms, from~\eqref{eqn:ljac} and  and Lemma~\ref{lem:xy} we have
		\begin{align}\label{eqn:lac}\begin{split}
				&\sum_{\substack{\Delta/H+c<a\le H\\0<c\le H}} J(a,c)\\
				&\quad =\sum_{0<c\le H}H\sum_{r|\Delta}\sum_{\substack{\Delta/H+c< a\le   H\\\gcd(a,c)=r}} \dfrac{r}{a} +O(H^{5/3+o(1)})\\
    &\quad = H\sum_{r|\Delta}\sum_{\substack{\Delta/H+c< a\le   H\\ 0<c\le H\\ \gcd(a,c)=r}} \dfrac{r}{a} +O(H^{5/3+o(1)})
  \\ &\quad = H\sum_{\substack{r|\Delta\\r\le H}} \left[\dfrac{6}{\pi^2} \dfrac{H}{r}+ O\left(\dfrac{\Delta^{1+o(1)}}{Hr}\right)\right] 
 + O(H^{5/3+o(1)})
 \\ &\quad = \dfrac{6}{\pi^2} H^2 \sum_{\substack{r|\Delta\\r\le H}}\dfrac{1}{r}+ O(H^{o(1)}\max(H^{5/3},\Delta)),
			    \end{split}
			\end{align}
			which completes the counting of the solutions of~\eqref{eqn:1-1} that also satisfies~\eqref{eqn:intla}.
			
			\subsection{Conclusion}
			We are now ready to prove Lemma~\ref{lem:sign} by counting $\#\cD_2^{1,1,-1}(H,\Delta)$. We recall that \begin{align}\label{eqn:sumjac}\begin{split}
					&\#\cD_2^{1,1,-1}(H,\Delta)
				= 
				\sum_{\substack{0<a\le \Delta/H+c\\0<c\le H}} J(a,c)+\sum_{\substack{\Delta/H+c<a\le H\\0<c\le H}} J(a,c).
			\end{split}\end{align}
  Then, substituting Equations~\eqref{eqn:sac} and~\eqref{eqn:lac} to~\eqref{eqn:sumjac}, we have \begin{align*}
  \begin{split}
					\#\cD_2^{1,1,-1}(H,\Delta)
&=\dfrac{12}{\pi^2} H^2 \sum_{\substack{r|\Delta\\r\le H}}\dfrac{1}{r}+ O(H^{o(1)}\max(H^{5/3},\Delta)),
\end{split}\end{align*}
which completes the proof of Lemma~\ref{lem:sign} for $\#\cD_2^{1,1,-1}(H,\Delta)$. This argument also proves Corollary~\ref{cor:tau}.
\appendix
\section{Discussions on a restricted divisor function}\label{apx}
This appendix serves to discuss several other properties of $\tau_N$. Even thought this function is relatively new in the literature, this object and the related problems are well-connected to several other known number theoretical problems.

We first recall that our main problem of counting $2\times 2$ nonsingular integer matrices is related to finding an asymptotic for the expression \begin{align*}
    \sum_{n\in \N}\tau_N(n)\tau_N(n+\Delta),
\end{align*} for a fixed $\Delta>0$ and as $N\to \infty$. 
With respect to the original divisor function $\tau$, we note the similarity of this expression and and the expression \begin{align}\label{eqn:additivedivisor}
    \sum_{n\le X} \tau(n)\tau(n+\Delta),
\end{align} for a fixed $\Delta>0$ and a parameter $X$. The problem of estimating~\eqref{eqn:additivedivisor} as $X\to \infty$  is precisely  the binary additive divisor problem, a topic that has been widely studied. In particular, it is known from Ingham~\cite{I} that as $X\to \infty$, \begin{align*}
     \sum_{n\le X} \tau(n)\tau(n+\Delta) \to \dfrac{6}{\pi^2} \dfrac{\sigma(\Delta)}{\Delta}  X\log^2 X.
\end{align*} For more details and improvements, we refer to Motohashi~\cite{Mo} and Meurman~\cite{Me}, as well as Balkanova and Frolenkov~\cite{BF}.
 
We also note that counting $2\times 2$ singular integer matrices is related to finding an asymptotic for the expression \begin{align*}
    \sum_{n\in \N}\tau_N^2(n).
\end{align*} In comparison with the function $\tau$, the corresponding problem is to find the second moment or mean square of the divisor function. The first result is due to Ramanujan~\cite{Ram}, who stated without proof the existence of a cubic polynomial $P(y)$ such that, as $X\to \infty$, \begin{align*}
    \sum_{n\le X} \tau^2(n) \to XP(\log X).
\end{align*} The best known error term of the above expression is due to Jia and Sankaranarayanan~\cite{JS}.

Even though the function $\tau_N$ is relatively new in the literature, this function can also be seen as a variant of the localised divisor function \begin{align*}
    \tau(n;y,z)=\#\{d\colon d|n, y<d\le z\}.
\end{align*}
In particular, we have \begin{align*}
    \tau_N(n) = \begin{cases}
        \tau(n;n/N,N),&\text{ if }n\text{ is not a square},
        \\
        \tau(n;n/N,N)+1,&\text{ if }n\text{ is a square}.
    \end{cases} 
\end{align*}
The function $\tau(n;y,z)$ has been extensively studied in the literature; e.g. see~\cite{Ford}.

We also note that a heuristic from~\cite{M,M2} implies that, roughly speaking, by assuming the set $\{\log d/\log N, d|N\}$ of flat quotients of the divisors of $N$ is uniformly distributed, we have \begin{align*}
    \tau_N(n) \approx \dfrac{\tau(n)}{\log N}.
\end{align*}

In addition, Mastrostefano~\cite{M2}'s note proved for any integer $k\ge 1$, \begin{align*}
    \sum_{n\le N^2}\tau_N(n)^k = D_kN^2(\log N)^{2^k-k-1} + O\left(N^2(\log N)^{2^k-k-2}\right) 
\end{align*} for some $D_k>0$, as $N\to \infty$. This equation on the $k$-th moment of $\tau_N$ also counts the number of integer solutions of the equation \begin{align*}
    d_1d_2=\dotsc = d_{2n-1}d_{2n},\qquad 1\le d_1\dotsc,d_{2n}\le N.
\end{align*}
We have \begin{align*}
    D_1=1,\quad D_2=\dfrac{12}{\pi^2}
\end{align*} from~\eqref{eqn:sumtaux} and Proposition~\ref{prop:tau}, respectively. In fact, the expression $D_k$ is a $2^k$-variable integral. For further discussion, we refer to~\cite[Equation~(3.11)]{M2}; see also~\cite{AC} and~\cite{dlB} for the general technique.

 \section*{Acknowledgements} The author would like to thank Alina Ostafe and Igor E. Shparlinski for supports and suggestions during the preparation of this work. The author would also like to thank Alexey Ustinov for discussion regarding the result in~\cite{Ust}, R\'{e}gis de la Br\`{e}teche for informing the author about~\cite{dlB}, Rachita Guria and Pieter Moree for informing the author about~\cite{GG,Guria}, and Marc Munsch for informing the author about~\cite{ACZ}. The author is supported by  Australian Research Council Grants  DP230100530 and a UNSW Tuition Fee Scholarship.

		\end{document}